\documentclass[twoside]{amsart}%

\usepackage{algorithm}
\usepackage{algpseudocode}
\usepackage{epigraph}
\usepackage{mathtools}
\usepackage{latexsym}
\usepackage{mathrsfs}
\usepackage{pstricks}
\usepackage{listings}
\usepackage{hyperref}
\usepackage{pgfplots}
\usepackage{appendix}
\usepackage{pst-all}
\usepackage{pstricks,pst-plot}
\pgfplotsset{compat=1.18} 
\usepackage{mathdots}
\usepackage{xcolor}
\usepackage{enumitem}
\usepackage{todonotes}
\usepackage{amssymb,bm}
\usepackage{amsmath}
\usepackage{amsfonts}
\usepackage{amsthm}
\usepackage{tikz}
\usepackage{epsfig}
\usepackage{verbatim}
\usepackage{float}
\usepackage{tabularx}
\usepackage{bm}
\usepackage{mathtools}
\usepackage{blkarray, bigstrut}
\usepackage{subcaption}

\definecolor{codegray}{rgb}{0.5,0.5,0.5}
\lstdefinestyle{mystyle}{
    commentstyle=\color{codegreen},
    keywordstyle=\color{magenta},
    numberstyle=\tiny\color{codegray},
    stringstyle=\color{codepurple},
    basicstyle=\ttfamily\footnotesize,
    breakatwhitespace=false,         
    breaklines=true,                 
    captionpos=b,                    
    keepspaces=true,                 
    numbers=left,                    
    numbersep=5pt,                  
    showspaces=false,                
    showstringspaces=false,
    showtabs=false,                  
    tabsize=2
}
  \lstdefinelanguage{GAP}{
    basicstyle=\ttfamily,
    keywords={true, false, function, return, fail, if, in, while, do, od, else, elif, fi, break, continue},
    keywordstyle=\color{blue}\bfseries,
    otherkeywords={
      >, <, ==
    },
    breaklines=true,      
    identifierstyle=\color{black},
    sensitive=True,
    comment=[l]{\#},
    commentstyle=\color{cyan},
    stringstyle=\color{red},
    morestring=[b]',
    morestring=[b]"
  }

\providecommand{\U}[1]{\protect\rule{.1in}{.1in}}

\newcolumntype{Y}{>{\raggedleft\arraybackslash}X}

\def\bn{{\mathbb{N}}}

\def\br{{\mathbb{R}}}

\def\bz{{\mathbb{Z}}}

\def\br{\mathbb R}

\def\vs{\vskip.3cm}
\usepackage{placeins}
\def\noi{\noindent}

\def\gdeg{G\text{\rm -deg}}
\def\o2mdeg{O(2)^m\text{\rm -deg}}

\def\s1deg{S^1\text{\rm -deg}}

\def\Om{\Omega}

\def\ve{\varepsilon}
\def\sign{\text{\rm sign\,}}

\def\ker{\text{\rm Ker\,}}

\DeclareMathOperator{\id}{Id}

\newcommand\bbR{\ensuremath{\mathbb R}}

\newrgbcolor{violet}{.6 .1 .8}
  \newrgbcolor{lightyellow}{1 1 .8}
  \newrgbcolor{lightblue}{.80 1 1}
  \newrgbcolor{mygreen}{0 .66 .05}
  \definecolor{mygreen}{rgb}{0,.66,.05}
  \definecolor{lightyellow}{rgb}{1,1,.80}
  \newrgbcolor{orange}{1 .6 0}
  \newrgbcolor{GreenYellow}{.85 1 .31}
  \newrgbcolor{Yellow}{1  1  0}
  \newrgbcolor{Goldenrod}{1  .90  .16}
  \newrgbcolor{Dandelion}{1  .71  .16}
  \newrgbcolor{Apricot}{1  .68  .48}
  \newrgbcolor{Peach}{1  .50  .30}
  \newrgbcolor{Melon}{1  .54  .50}
  \newrgbcolor{YellowOrange}{1  .58  0}
  \newrgbcolor{Orange}{1  .39  .13}
  \newrgbcolor{BurntOrange}{1  .49  0}
  \newrgbcolor{Bittersweet}{1.  .4300  .24}
  \newrgbcolor{RedOrange}{1  .23  .13}
  \newrgbcolor{Mahogany}{1.  .4475  .4345}
  \newrgbcolor{Maroon}{1.  .4084  .5376}
  \newrgbcolor{BrickRed}{1.  .3592  .3232}
  \newrgbcolor{Red}{1  0  0}
  \newrgbcolor{OrangeRed}{1  0  .50}
  \newrgbcolor{RubineRed}{1  0  .87}
  \newrgbcolor{WildStrawberry}{1  .04  .61}
  \newrgbcolor{CarnationPink}{1  .37  1}
  \newrgbcolor{Salmon}{1  .47  .62}
  \newrgbcolor{Magenta}{1  0  1}
  \newrgbcolor{VioletRed}{1  .19  1}
  \newrgbcolor{Rhodamine}{1  .18  1}
  \newrgbcolor{Mulberry}{.6668  .1180  1.}
  \newrgbcolor{RedViolet}{.9538  .4060  1.}
  \newrgbcolor{Fuchsia}{.5676  .1628  1.}
  \newrgbcolor{Lavender}{1  .52  1}
  \newrgbcolor{Thistle}{.88  .41  1}
  \newrgbcolor{Orchid}{.68  .36  1}
  \newrgbcolor{DarkOrchid}{.60  .20  .80}
  \newrgbcolor{Purple}{.55  .14  1}
  \newrgbcolor{Plum}{.50  0  1}
  \newrgbcolor{Violet}{.98 .15 .95}
  \newrgbcolor{RoyalPurple}{.25  .10  1}
  \newrgbcolor{BlueViolet}{.84  .38  .98}
  \newrgbcolor{Periwinkle}{.43  .45  1}
  \newrgbcolor{CadetBlue}{.38  .43  .77}
  \newrgbcolor{CornflowerBlue}{.35  .87  1}
  \newrgbcolor{MidnightBlue}{.4414  .9259  1.}
  \newrgbcolor{NavyBlue}{.06  .46  1}
  \newrgbcolor{RoyalBlue}{0  .50  1}
  \newrgbcolor{Blue}{0  0  1}
  \newrgbcolor{Cerulean}{.06  .89  1}
  \newrgbcolor{Cyan}{0  1  1}
  \newrgbcolor{ProcessBlue}{.04  1  1}
  \newrgbcolor{SkyBlue}{.38  1  .88}
  \newrgbcolor{Turquoise}{.15  1  .80}
  \newrgbcolor{TealBlue}{.1572  1.  .6668}
  \newrgbcolor{Aquamarine}{.18  1  .70}
  \newrgbcolor{BlueGreen}{.15  1  .67}
  \newrgbcolor{Emerald}{0  1  .50}
  \newrgbcolor{JungleGreen}{.01  1  .48}
  \newrgbcolor{SeaGreen}{.31  1  .50}
  \newrgbcolor{Green}{0  1  0}
  \newrgbcolor{ForestGreen}{.1992  1.  .2256}
  \newrgbcolor{PineGreen}{.3100  1.  .5575}
  \newrgbcolor{LimeGreen}{.50  1  0}
  \newrgbcolor{YellowGreen}{.56  1  .26}
  \newrgbcolor{SpringGreen}{.74  1  .24}
  \newrgbcolor{OliveGreen}{.6160  1.  .4300}
  \newrgbcolor{RawSienna}{.53  .28  .16}
  \newrgbcolor{Sepia}{1.  .7510  .70}
  \newrgbcolor{Brown}{.41  .25  .18}
  \newrgbcolor{TAN}{.86  .58  .44}
  \newrgbcolor{Gray}{1.  1.  1.}
  \newrgbcolor{Black}{1  1  1}
  \newrgbcolor{White}{1  1  1}

\newtheorem{theorem}{Theorem}[section]

\newtheorem{lemma}{Lemma}[section]

\newtheorem{remark}{Remark}[section]

\newtheorem{remark-definition}{Remark and Definition}[section]
\newtheorem{rem-not}{Remark and Notation}[section]

\begin{document}

\title[Non-Radial Solutions to Elliptic Systems]{Unbounded Branches of Non-Radial Solutions to Semilinear Elliptic Systems on a Unit Ball in $\mathbb R^3$ and Their Patterns}

\author[C. Crane --- Z. Ghanem]{Casey Crane --- Ziad Ghanem}

\address
{\textsc{Casey Crane}\\
Department of Mathematical Sciences\\
University of Texas at Dallas\\
Richardson, TX 75080, USA.}
\email{casey.crane@utdallas.edu}

\address
{\textsc{Ziad Ghanem}\\
Department of Mathematical Sciences\\
University of Texas at Dallas\\
Richardson, TX 75080, USA}

\email{ziad.ghanem@utdallas.edu}



\begin{abstract}
We investigate symmetry-breaking phenomena in semilinear elliptic systems on the unit ball in $\mathbb{R}^3$, focusing on the emergence of non-radial solution branches with prescribed spatial and internal symmetries. Extending previous scalar results, we develop a framework for systems equivariant under $G := O(3) \times \Gamma \times \mathbb{Z}_2$, where $\Gamma$ is a finite group encoding coupling symmetries. Using the $G$-equivariant Leray--Schauder degree and Burnside ring techniques, we derive computable criteria for the existence of unbounded branches of non-radial solutions and classify their isotropy types. Our approach accommodates non-simple eigenvalue multiplicities and provides explicit bifurcation conditions in terms of spectral resonance between coupling eigenvalues and spherical Laplacian modes. Applications to coupled spherical oscillators illustrate how Platonic symmetries and internal permutations interact to produce complex patterns. These results establish a general method for detecting and characterizing symmetry-breaking in high-dimensional elliptic systems.
\end{abstract}

\maketitle

\setlength{\epigraphwidth}{.9\textwidth}
\epigraph{So much for their passage into one another: I must now speak of their construction... And there is a fifth figure (which is made out of twelve pentagons), the dodecahedron—this God used as a model for the twelvefold division of the Zodiac.}{Plato, Timaeus 54d–55c}
\section{Introduction}
Seeking geometric order and regularity in a changeful world, the ancient Greeks were captivated by patterns in nature. Among these patterns, the sphere held a privileged place: a symbol of perfection and harmony, celebrated not only for its radial symmetry but for its ability to host the most elegant of configurations—the Platonic solids. These polyhedra, with symmetry groups $\mathbb{T}$ (tetrahedral), $\mathbb{O}$ (octahedral), and $\mathbb{I}$ (icosahedral), represent some of the most fundamental non-radial patterns realizable on a spherical domain. The ancient fascination with symmetry, once rooted in philosophical and aesthetic reflections, eventually gave rise to a rigorous mathematical theory. With origins in nineteenth-century crystallographical classifications, symmetry has since evolved into a formal concept expressed through the language of transformation groups \cite{Bruckler}. 
\vs
Today, the symmetry perspective permeates analysis, where one of the most fundamental questions concerns the relationship between the symmetries of a domain and those expressed by the solutions to a differential equation. In the context of elliptic partial differential equations, this line of inquiry was profoundly shaped by the American mathematician James Serrin who, in 1971, considered a problem in potential theory involving an overdetermined boundary condition \cite{Serrin}. Serrin demonstrated that, if a solution to an elliptic equation on a domain with a smooth boundary simultaneously satisfies constant Dirichlet and Neumann conditions, then the domain must be a ball and any solution must be radially symmetric. His proof introduced a moving planes technique, where a plane is moved to a critical position and the maximum principle is used to infer symmetric invariance. Almost a decade later, Gidas, Ni and Nirenberg adapted and expanded this technique to prove that any positive, twice continuously differentiable solution to a Dirichlet problem of the form $\Delta u + f(u) = 0$ on a ball must be radially symmetric \cite{Gidas}. Their work established that, even without an overdetermined boundary condition, the combination of a symmetric domain, solution positivity and regularity is sufficient to guarantee radial symmetry for a wide class of scalar equations. Dancer later demonstrated the remarkable robustness of this result in \cite{Dancer} by adapting the moving planes technique for domains with non-smooth boundaries and weaker solution classes to prove that positivity alone suffices to guarantee radial symmetry, even without the strict regularity assumptions required by Gidas et al. in \cite{Gidas}.
\vs
Together, these foundational results establish a clear paradigm: for a single scalar equation under conditions of positivity, solutions on a ball are forced to inherit the domain's full radial symmetry. This naturally raises the question of what happens when these restrictive hypotheses are relaxed. 
Smoller and Wasserman addressed this question directly in \cite{Smoller}, 
proving that a family of elliptic equations on a ball can only exhibit symmetry breaking when the associated linearized operator develops a nontrivial kernel containing non-radial functions.
Their analysis, which relied on a bifurcation theorem from Vanderbauwhede's monograph \emph{Local Bifurcation and Symmetry} \cite{Vanderbauwhede}, revealed the possibility of entire manifolds of patterned solutions emerging from these bifurcation points. Building on this work and answering a conjecture posed therein, Giovanna Cerami provided further sufficient conditions for symmetry breaking to occur for a similar class of problems in \cite{Cerami}. Cerami's analysis not only confirmed the local existence of bifurcating non-radial solutions but also investigated the global structure of the resulting solution set, showing how branches of non-symmetric solutions might extend away from a bifurcation point.
\vs
While the theory for scalar equations is now well-developed, the situation becomes considerably more interesting when we turn to systems of elliptic equations, where the interplay between multiple solution components can lead to fundamentally new phenomena that have no scalar analog. The coupling between equations may either stabilize radial solutions against symmetry breaking or, conversely, create new mechanisms for bifurcation that are impossible in the one-component setting. Of particular interest are systems where the coupling structure has a prescribed symmetry group. Consider, for example, a configuration of $N$ coupled components whose interactions are governed by a finite group $\Gamma$ acting via permutation on $\br^N$ as follows
\[
\sigma (u_1,\ldots,u_N) := (u_{\sigma(1)}, \ldots, u_{\sigma(N)}), \quad (u_1,\ldots,u_N) \in \br^N, \; \sigma \in \Gamma.
\]
To the best of our knowledge, a comprehensive analysis of the possible non-radial patterns for a general elliptic system on the sphere, accommodating arbitrary coupling symmetries, has not yet been undertaken. As we will demonstrate, this setting allows for an exceptionally rich structure of solutions, where the spatial symmetries, suggested by the Platonic solids, interact with the internal symmetries of the coupling to produce complex, multi-component patterns.
\subsection{The Governing Equations and their Symmetries}
Motivated by these considerations, we propose to model the dynamics of $N$ identical coupled spherical oscillators with the following parameter-dependent system of semilinear elliptic equations on the three-dimensional unit ball $\Omega := \{x \in \br^3 : |x|<1\}$:
\begin{align}\label{eq:system}
    \begin{cases}
        -\Delta u = f(x,u) + A(\alpha)u, \quad u(x) \in \br^N; \\
        u|_{\partial \Om} = 0,
    \end{cases}
\end{align}
where $A:\br \rightarrow L^\Gamma(\br^N)$ is a family of $\Gamma$-equivariant $N \times N$ matrices depending continuously on the parameter $\alpha \in \br$ and $f: \overline{\Om} \times \br^N \rightarrow \br^N$ is a continuous map satisfying the conditions:
\begin{enumerate}[label=($A_\arabic*$)]
\item\label{a1} $f(\gamma x, u) = f(x,u)$ for all $x \in \Om$, $u \in \br^N$ and $\gamma \in O(3)$; 
\item\label{a2} $f(x,\sigma u) = \sigma f(x,u) $ for all $x \in \Om$, $u \in \br^N$ and $\sigma \in \Gamma$; 
\item\label{a3} $f(x, -u) = -f(x,u)$ for all $x \in \Om$ and $u \in \br^N$; 
\item\label{a4} $f(u)$ is $o(|u|)$ as $u$ approaches $0$, i.e.
\[
    \lim_{u \rightarrow 0} \frac{f(x,u)}{|u|} = 0.
\]
\item\label{a5} there exist numbers $a,b >0$ and $c \in (0,1)$ such that
\[
|f(x,u)| < a|u|^c + b, \quad u \in \br^N.
\]
\end{enumerate}
Assumptions \ref{a1}--\ref{a3} ($O(3)$-invariance, $\Gamma$-equivariance and oddness, respectively) imply that the system \eqref{eq:system} admits the symmetry group $G:= O(3) \times \Gamma \times \bz_2$ under the standard action of $O(3)$ on $\br^3$,
the permutation action of $\Gamma$ on $\br^N$, 
and the antipodal $\bz_2$-action $u \rightarrow -u$. Assumption \ref{a4} ensures that the linear part of \eqref{eq:system} near the trivial solution $u \equiv 0$ is given by the matrix $A(\alpha): \br^N \rightarrow \br^N$. Finally, the sublinear growth 
Assumption \ref{a5} is necessary for establishing a priori bounds on solutions to \eqref{eq:system}.
\vs
The zero function $u \equiv 0$ satisfies the boundary value problem \eqref{eq:system} for every parameter value $\alpha \in \br$. Nontrivial solutions to \eqref{eq:system} with $u \not\equiv 0$ may arise via bifurcation from the trivial branch $\{(\alpha,0) : \alpha \in \br\}$. Adopting standard spherical coordinate notation 
$(r,\theta,\phi)$, where $r = |x| \in [0,1]$ is the radius, $\theta \in [0,2\pi]$ is the polar inclination in the $xy$-plane and $\phi \in [0,\pi]$ is the azimuthal angle measured from the positive $z$-axis, these non-zero solutions can be classified according to their spatial symmetries as follows:
\begin{itemize}
\item[(i)] {\bf Radial solutions}, depending only on the radius $r = |x|$, are invariant under all spatial rotations and point reflections, i.e. $u(\gamma x) = u(x)$ for all $\gamma \in O(3)$.
These solutions can be identified using classical methods (e.g. by reducing \eqref{eq:system} to a second order ODE in the variable $r$).
\item[(ii)] {\bf Non-radial solutions}, depending nontrivially on the azimuthal angle $\phi$ and the polar inclination $\theta$, admit isotropy subgroups $G_u \leq G$ not containing the radial subgroup $O(3) \times \bz_1 \times \bz_1$. These solutions might exhibit exotic symmetries which enhance our understanding of the model.
\end{itemize}
Our primary goal is to demonstrate how the $G$-equivariant Leray-Schauder degree can be used to derive practical criteria guaranteeing the existence of unbounded branches of non-radial solutions with prescribed symmetries. 
The emergence of a branch of non-radial solutions from the trivial solution $(\alpha,0)$ is only possible at critical parameter values $\alpha_0 \in \br$ for which the linearization of \eqref{eq:system} about $(\alpha_0,0)$ becomes singular. Singularity occurs under a resonance condition
between the spectrum $\sigma(A(\alpha)) =\{\mu_1(\alpha),\mu_2(\alpha),\ldots, \mu_r(\alpha)\}$ and the positive zeros $\{s_{m,n}\}_{n \in \bn}$ of the spherical Bessel functions of the first kind $\{J_m(x)\}_{m \geq0}$. Specifically, we will find that $\alpha_0 \in \br$ is a critical parameter value for \eqref{eq:system} if there exists an index triple  $(m,n,j)\in \bn \cup \{0\} \times \bn \times \{1,2,\ldots,r\}$ for which $\mu_j(\alpha_0) = s_{m,n}$.
\subsection{Theoretical Framework and a Refined Isotropy Classification}
The effectiveness of the $G$-equivariant Leray-Schauder for detecting the existence of non-radial solutions to a boundary value problem equivalent to \eqref{eq:system} with trivial coupling, i.e. $\Gamma = \bz_1$, at any parameter value $\alpha \in \br$ satisfying $\mu_j(\alpha) \neq s_{m,n}$ for all $(m,n,j) \in \bn \cup \{0\} \times \bn \times \{1,2,\ldots,r\}$ has already been demonstrated by Liu et al. in \cite{Jingzhou}. Under the additional assumption that the matrix $A(\alpha)$ admits only simple eigenvalues, 
a sufficient condition for the existence of solutions admitting an isotropy related to a subgroup $H \leq G$ implying non-radial symmetries is established. Our paper is a natural extension of this framework to the bifurcation problem \eqref{eq:system}, allowing for non-simple eigenvalue multiplicities---an important consideration in bifurcation analysis where multiple eigenvalues might cross zero simultaneously.
\vs
Our approach requires reformulation of the bifurcation problem \eqref{eq:system} in an appropriate functional space $\mathscr H:= H_0^2(\Om; \br^N)$ as a fixed point equation for the nonlinear operator 
\[
\mathscr F:\br \times \mathscr H \rightarrow \mathscr H, \quad \mathscr F(\alpha,u) := u - \mathscr L^{-1}(N_f(j(u)) + A(\alpha) j(u)),
\]
where $\mathscr L: \mathscr H \rightarrow L^2(\Om;\br^N)$ is the Laplacian operator $\mathscr Lu := - \Delta u$, $N_f : L^q(\Om;\br^N) \rightarrow L^2(\Om;\br^N)$ is the superposition operator $N_f(v)(x) := f(x,v(x))$ and $j: \mathscr H \hookrightarrow L^q(\Om;\br^N)$ is the embedding $j(u(x)) := u(x)$, these last two defined for any $q > \max\{1,2c\}$ (cf. Assumption \ref{a5}). Following \cite{Jingzhou},  we note that $(\rm i)$ $\mathscr F$ is a completely continuous $G$-equivariant field (a necessary prerequisite for the application of the $G$-equivariant Leray-Schauder degree), $(\rm ii)$ the operator $\mathscr F$ is differentiable at the origin and $(\rm iii)$ for a sufficiently large $R > 0$ and a sufficiently small $\varepsilon > 0$, the annulus $B_R(0)\setminus B_\varepsilon(0)$ contains every solution to \eqref{eq:system}. Under the additional assumption that $D\mathscr F(\alpha,0): \mathscr H \rightarrow \mathscr H$ is an isomorphism, the existence of nontrivial solutions to \eqref{eq:system} is therefore equivalent to the nontriviality of the topological invariant 
\[
\gdeg(\mathscr F(\alpha),B_R(0)\setminus B_\varepsilon(0)) = (G) - \gdeg(D\mathscr F(\alpha,0), B(\mathscr H)),
\]
where $B(\mathscr H) := \{u \in \mathscr H : \|u\|_{\mathscr H} < 1\}$, $\gdeg$ is the $G$-equivariant Leray-Schauder degree and $(G) \in A(G)$ is the unit element of the Burnside ring (see Appendix \ref{sec:appendix} for definition of the Burnside ring $A(G)$, definition of $G$-admissibility/the set of all admissible $G$-pairs $\mathcal M^G$ and for definition of the $G$-equivariant Leray-Schauder degree $\gdeg:\mathcal M^G \rightarrow A(G)$). In this way, the problem of finding nontrivial solutions to \eqref{eq:system} is recast as the problem of finding non-unit orbit types $(H) \in \Phi_0(G) \setminus \{(G)\}$ appearing nontrivially in the Burnside ring element $\gdeg(D\mathscr F(\alpha,0), B(\mathscr H)) \in A(G)$:
\begin{lemma}\label{lemm:sufficient_condition}
Let $\alpha \in \mathbb{R}$ be such that $D\mathscr F(\alpha,0)$ is an isomorphism. If there exists a \textbf{non-unit} orbit type $(H) \in \Phi_0(G) \setminus \{(G) \}$ for which one has
\[
\operatorname{coeff}^{H}\left( \gdeg(D\mathscr F(\alpha,0), B(\mathscr H)) \right) \neq 0,
\]
then there exists a \textbf{nontrivial} solution $(\alpha, u) \in \br \times \mathscr H$  to \eqref{eq:system} with an isotropy subgroup $G_u \leq G$ satisfying the relation $(G_u) \geq (H)$. 
\end{lemma}
Alongside \cite{Jingzhou}, we 
leverage the description of the list of all irreducible $G$-representations \\ $\{\mathcal V_{m,j}\}_{m \geq 0, \; j \in \{1,\ldots,r\}}$ provided by Golubitsky in \cite{Golubitsky} (adapted in the standard way from $O(3)$ to the product group $O(3) \times \Gamma \times \bz_2$, see Appendices \ref{app:amalgamated_notation} and \ref{app:max_orbit_types} for more details) to describe the
$G$-isotypic decomposition of $\mathscr H$ and identify the spectrum of the linearized operator $D\mathscr F(\alpha,0)$, which are then used to derive a computational formula for $\gdeg(D\mathscr F(\alpha,0), B(\mathscr H))$ in terms of the Burnside ring product of a finite number of basic degrees $\deg_{\mathcal V_{m,j}} := \gdeg(-\id, B(\mathcal V_{m,j}))$ whose contribution is determined by the isotypic multiplicities $m_j \in \bn$ of the eigenvalues $\mu_j(\alpha)$ and the negative spectrum of $D\mathscr F(\alpha,0)$---which we account for with the index set $\Sigma(\alpha) := \{ (m,n,j) \in \bn \cup \{0\} \times \bn \times \{1,2,\ldots,r\} : \mu_j(\alpha) > s_{m,n} \}$ as follows:
\begin{align} \label{eq:computational_formula_gdegA_intro}
    \gdeg(D\mathscr F(\alpha,0), B(\mathscr H)) \; = \prod\limits_{(m,n,j) \in \Sigma(\alpha)} (\deg_{\mathcal V_{m,j}})^{m_j}.
\end{align}
Extracting specific symmetry information 
from the computational formula
\eqref{eq:computational_formula_gdegA_intro} via Lemma  \ref{lemm:sufficient_condition} requires careful application of analytic tools related to the $G$-equivariant degree, the multiplicative structure of the Burnside ring and the natural ordering relations in the isotropy lattice $\Phi_0(G;\mathscr H\setminus\{0\})$ (cf. Appendix \ref{sec:appendix}). In principle, the coefficients standing next to each orbit type in the basic degree $\deg_{\mathcal V_{m,j}} \in A(G)$ can only be identified via the recurrence formula for the $G$-equivariant Brouwer degree with complete knowledge of the coefficients standing next to orbit types $(K) \in \Phi_0(G; \mathcal V_{m,j} \setminus \{0\}$) satisfying $(K) \geq (H)$ (cf. Appendix \ref{sec:appendix}, Eq. \eqref{eq:RF_bd}). It follows that, in practice, the only non-unit orbit types $(H)$ which are known to appear nontrivially in the basic degree $\deg_{\mathcal V_{m,j}}$ are those which are maximal in the corresponding sublattice $\Phi_0(G; \mathcal V_{m,j} \setminus \{0\})$ and admit an odd-dimensional $H$-fixed point space $\mathcal V_{m,j}^{H}$. Indeed, the coefficient $n_H := \operatorname{coeff}^H(\deg_{\mathcal V_{m,j}})$ associated with any orbit type $(H)$ which appears maximally in $\Phi_0(G; \mathcal V_{m,j} \setminus \{0\})$ is given by the non-recursive relation
\[
n_H = \frac{(-1)^{\dim \mathcal V_{m,j}^H} - 1}{|W(H)|}.
\]
On the other hand, the recurrence formula for multiplication in the Burnside ring (cf. Appendix \ref{sec:appendix}, Eq. \ref{def:recurrence_formula_coefficients_burnside_product}) dictates the behavior of orbit types in the Burnside ring product of two or more basic degrees. In general, only the coefficients associated with \textbf{maximal} orbit types, i.e. orbit types $(H)$ which are maximal in the full isotropy lattice $\Phi_0(G; \mathscr H \setminus \{0\})$, can be readily computed in an arbitrary product of basic degrees, for which the recursive formula simplifies to
\begin{align} \label{eq:product_rule_maximalorbitypes}
    \operatorname{coeff}^H((K) \cdot (H)) =
    \begin{cases}
        1 & \text{ if } (K) = (G); \\
        |W(H)| & \text{ if } (K) = (H); \\
        0 & \text{ otherwise.}
    \end{cases}
\end{align}
In \cite{Jingzhou}, Liu et al. introduce the terminology {\it maximal non-radial} for orbit types which appear as maximal elements of an isotropy sublattice $\Phi_0(G; \mathcal V_{m,j} \setminus \{0\})$ for some $m \in 2\bn - 1$, implicitly assuming that this class of orbit types admits maximality in the full lattice $\Phi_0(G)$ by using the simplified rule \eqref{eq:product_rule_maximalorbitypes} when making their Burnside ring product computations. However, a particular family of orbit types $(D_{2m}^p {}^{D_{2m}^d}\times_{\bz_2} \bz_2)$, which appear maximally in the corresponding isotropy sublattices $\Phi_0(G; \mathcal V_{m,j} \setminus \{0\})$, can be shown to satisfy $(D_{2m}^p {}^{D_{2m}^d}\times_{\bz_2} \bz_2) \leq (D_{2m'}^p {}^{D_{2m'}^d}\times_{\bz_2} \bz_2)$ for an infinite number of $m' \neq m$ 
(see example \cite{Ghanem1} where the relations between analogous orbit types in the isotropy lattice $\Phi_0(O(2) \times \bz_2)$ are explored). On the other hand, in the setting of $\Gamma = \bz_1$, there is a maximal orbit type $(\mathbb O^p {}^{\mathbb T^p}\times_{\bz_2} \bz_2) \in \Phi_0(G; \mathscr H \setminus \{0\})$ present in the isotropy sublattice $\Phi_0(G; \mathcal V_{m,j} \setminus \{0\})$ for some $m \in 2 \bn$ which corresponds to non-radial symmetries. Although the rest of their analysis holds and serves as the basis for the main results derived in this paper, the above considerations motivate a reclassification of $\Phi_0(G)$. For us, an orbit type $(H)$ is said to be {\bf maximal non-radial} if it is \textbf{maximal} and the \textbf{radial} subgroup $O(3) \times \bz_1 \times \bz_1 \leq G$ is not subconjugate to $H$.
\vs
Outfitted with this new terminology, we consider the following refinement of the Lemma \ref{lemm:sufficient_condition}:
\begin{lemma}\label{lemm:sufficient_condition_nonradial}
Let $\alpha \in \mathbb{R}$ be such that $D\mathscr F(\alpha,0)$ is an isomorphism. If there exists a \textbf{maximal non-radial} orbit type $(H) \in \Phi_0(G)$ for which one has
\[
\operatorname{coeff}^{H}\left( \operatorname{\gdeg}(D\mathscr F(\alpha,0), B(\mathscr H)) \right) \neq 0,
\]
then there exists a \textbf{non-radial} solution $(\alpha,u) \in \br \times \mathscr{H}$ to \eqref{eq:system} with an isotropy subgroup $G_u \leq G$ satisfying $(G_u) = (H)$.
\end{lemma}
\begin{proof}
The non-zero coefficient 
standing next to $(H)$
guarantees a nontrivial solution $u \in \mathscr H \setminus \{0\}$ with an isotropy satisfying $G_u \ge H$ by the existence property of the $G$-equivariant Leray-Schauder degree (cf. Appendix \ref{sec:appendix}). On the other hand, since by Lemma \ref{lemm:sufficient_condition} one has $u \not\equiv 0$, it cannot be that $(G_u) = (G)$ and the conclusion follows.
\end{proof}
\vs
\subsection{Organization} The remainder of this paper is organized as follows. In Section \ref{sec:functional_space}, we prepare the bifurcation problem \eqref{eq:system} for application of the $G$-equivariant Leray-Schauder degree by a functional reformulation in the Sobolev space $\mathscr H$ with a nonlinear operator $\mathscr F$ in the form of a $G$-equivariant compact perturbation of identity. In Section \ref{sec:isotypic_decomp}, we consider the $G$-isotypic decomposition of $\mathscr H$, obtain the spectrum of the linearization $D \mathscr F(\alpha,0)$ and derive a computational formula for the degree invariant $\gdeg(D \mathscr F(\alpha,0), B(\mathscr H))$. In Section \ref{sec:existence_result}, we establish our main existence result for non-radial solutions to \eqref{eq:system} at any fixed regular parameter value $\alpha \in \br$ by 
demonstrating how the coefficients associated with any maximal non-radial orbit type $(H)$ in the $G$-degree $\gdeg(D\mathscr F(\alpha,0), B(\mathscr H))$ is determined according to the parity of the number of eigenvalues $\mu_j(\alpha) \in \sigma(A(\alpha))$,
up to isotypic (geometric) multiplicity $m_j$, which dominate the positive zeros $s_{m,n}$ of the Bessel functions of the first kind associated with irreducible $G$-representations $\mathcal V_{m,j}$ admitting odd-dimensional $H$-fixed point spaces, i.e. by the quantity 
\[
\mathfrak m^H(\alpha) := \left| (m,n,j) \in \Sigma(\alpha) :  2 \nmid m_j, \; 2 \nmid \dim\mathcal V_{m,j}^H  \}\right|.
\]
This provides the foundation for our bifurcation analysis in Section \ref{sec:bifurcation_results}, where we prove that a change in parity of $\mathfrak m^H(\alpha)$ as $\alpha$ crosses an isolated critical parameter value $\alpha_0$ guarantees the emergence of a branch of non-radial solutions emerging from the trivial solution branch at $(\alpha_0,0) \in \br \times \mathscr H$ with symmetry $(H)$. In Section \ref{sec:rab_alt}, we present our main global bifurcation result under the additional assumption that our set of critical parameter values is finite:
\begin{theorem}\label{thm:main_global_bif}
Suppose that the set of critical parameters for \eqref{eq:system} can be enumerated 
\[
\{ \alpha_1, \alpha_2, \ldots \alpha_n \} \subset \br,
\]
in such a way that, for any $i < j$, one has $\alpha_i < \alpha_j$. If for any  maximal non-radial orbit type $(H)$ the parities of $\mathfrak m^H(\alpha_1^-)$ and $\mathfrak m^H(\alpha_n^-)$ disagree, then the system \eqref{eq:system} admits an unbounded branch of non-radial solutions emerging from the trivial solution with symmetry $(H)$.
\end{theorem}
Finally, in Section \ref{sec:example}, to demonstrate the versatility of our framework, we consider its application to a system of two spherical oscillators with coupling symmetries $S_2$ modeled by \eqref{eq:system} with the family of symmetric coupling matrices 
\begin{equation}\label{eq:model-A}
A(\alpha)=\zeta(\alpha)(a\,I+b\,P),\qquad P=\begin{pmatrix}0&1\\1&0\end{pmatrix},
\end{equation}
specified by a pair of model variables $a,b\in\br$ and a continuous profile (see Figure \ref{Fig:sigmoid}) given by the sigmoid function 
\begin{align} \label{eq:sigmoid}
    \zeta(\alpha) = \frac{1}{1 + e^{-\alpha}}.
\end{align}
\begin{figure}[htbp]
    \centering
\includegraphics[width=.9\textwidth]{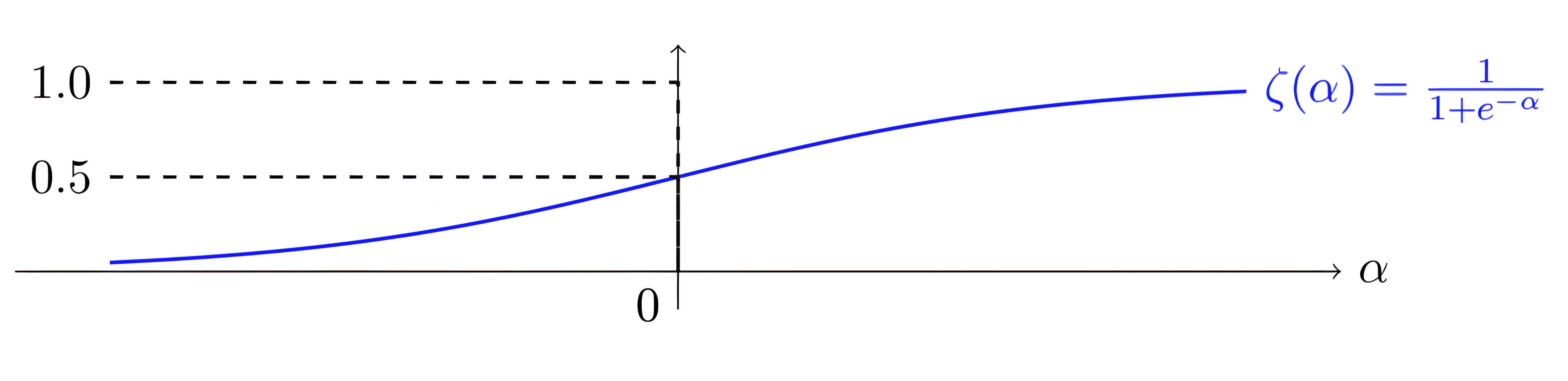}
    \caption{Graph of the sigmoid function $\zeta(\alpha)$.}
    \label{Fig:sigmoid}
\end{figure}
After systematically deriving the set of maximal non-radial orbit types $\{ O(2)^-_j, \mathbb O_j, \mathbb O^-_j, \mathbb O^+_j, \mathbb I_j : j = 0,1 \}$ (see Appendix \ref{app:max_orbit_types}, Table \ref{table:max-orbit-types-abbrev}) in the isotropy lattice $\Phi_0(G)$ and determining the relevant Fourier indices (see Appendix \ref{app:max_orbit_types}, Table \ref{table:max-orbit-types-sequential})
\begin{align*}
	\mathcal{M}_H :=
	\begin{cases}
		\{m \in \bn : 1 \leq m \leq 59 \text{ and } m \text{ is odd} \} & \text{if } H = O(2)^-_j; \\
		 \{3, 7, 9, 11, 13, 17, 21,25, 27, 31, 33, 35, 37, 41, 51, 55, 57, 59\}  & \text{if } H = \mathbb{O}^-_j;  \\
		\{9, 13, 15, 17, 19, 23, 33, 37, 39, 41, 43, 47, 57, \dots\}  & \text{if } H = \mathbb{O}_j; \\
		\{6, 10, 12, 14, 16, 20, 30, 34, 36, 38, 40, 44, 54, 58, 60\} & \text{if } H = \mathbb{O}^+_j; \\
		\{15, 21, 25, 27, 31, 33, 35, 37, 39, 41, 43, 47, 49, 53, 59\}  & \text{if } H = \mathbb{I}_j,
	\end{cases}
\end{align*}
for which these orbit types appear in the sublattices $\Phi_0(G; \mathscr H_{m,j})$, we apply Theorem \ref{thm:main_global_bif} to obtain the following practical criterion for global bifurcation of non-radial, patterned solution branches to \eqref{eq:system}:
\begin{theorem}\label{thm:exampleo2}
Let $(H)$ be any maximal non-radial orbit type from the selection
\[
H \in \{ O(2)^-_j, \mathbb O_j, \mathbb O^-_j, \mathbb O^+_j, \mathbb I_j : j = 0,1 \}.
\]
The system \eqref{eq:system} with the family of coupling matrices $A(\alpha) : \br^2 \rightarrow \br^2$ given by \eqref{eq:model-A} admits an unbounded branch of non-radial solutions with symmetry $(H)$ if the quantity
\[
\sum_{\substack{m \equiv k \pmod{60} \\ \text{for some } k \in \mathcal M_H }} \bigg|\{n\ge1:s_{m,n}<a+ (-1)^jb\}\bigg|,
\]
is an \textbf{odd number}.
\end{theorem}
\section{Functional Space Reformulation} \label{sec:functional_space}
Following \cite{Jingzhou}, we consider the Sobolev space $\mathscr H:=H_{0}^2(\Om,\br^N)$ equipped with the usual norm $\|u\|:=\max\{\|D^s u\|_{L^2}: |s|\le 2\}$, where $s:=(s_1,s_2,s_3)$ is a multi-index and $D^s$ is the partial derivative operator with respect to Cartesian coordinates $(x,y,z) \in \br^3$. Functions in $\mathscr H$ can be represented using a basis derived from the eigenfunctions of the 
spherical Laplacian operator 
\[
\mathscr L:\mathscr H\to L^2(\Om;\br^N), \quad \mathscr Lu:=-\triangle u,
\]
which admits the spectrum $\sigma(\mathscr L) := \{s_{m,n} : m = 0,1,\ldots, \; n = 1,2,\ldots \}$
and the corresponding eigenspaces (see Figure \ref{fig:eigenfunctions})
\begin{align}\label{def:space_Emn}
 \mathscr E_{m,n} := \left\{  J_m(r \sqrt{s_{m,n}})\left(\cos(k \theta)P_m^k(\cos(\phi))a_k + \sin(k \theta)P_m^k(\cos(\phi))b_k\right) : 0 \leq k \leq m, \; a_k,b_k \in \br^N \right\},   
\end{align}
where $J_m:\br \rightarrow \br$ is the $m$-th spherical Bessel function of the first kind, $\sqrt{s_{m,n}}$ is its $n$-th positive zero, $P_m: \br \rightarrow \br$ is the $m$-th Legendre polynomial and, for each $0 \leq k \leq m$, the notation $P_m^k$ is used to indicate the polynomial $P_m^k(x):= (1-x^2)^{k/2} \partial_x^k P_m(x)$. 
Specifically, every function $u \in \mathscr H$ admits a Fourier expansion of the form
\[
u(r,\theta, \phi) = \sum_{m=0}^\infty \sum_{n = 1}^{\infty} \sum_{k = 0}^{m} J_m(r \sqrt{s_{m,n}})\left(\cos(k \theta)P_m^k(\cos(\phi))a_{m,n,k} + \sin(k \theta)P_m^k(\cos(\phi))b_{m,n,k}\right), \quad 
\]
where $a_{m,n,k},b_{m,n,k} \in \br^N$ and $b_{0,n,0} = 0$ for all $n \in \bn$.
\begin{figure}[H]
    \centering
    \includegraphics[width=\linewidth]{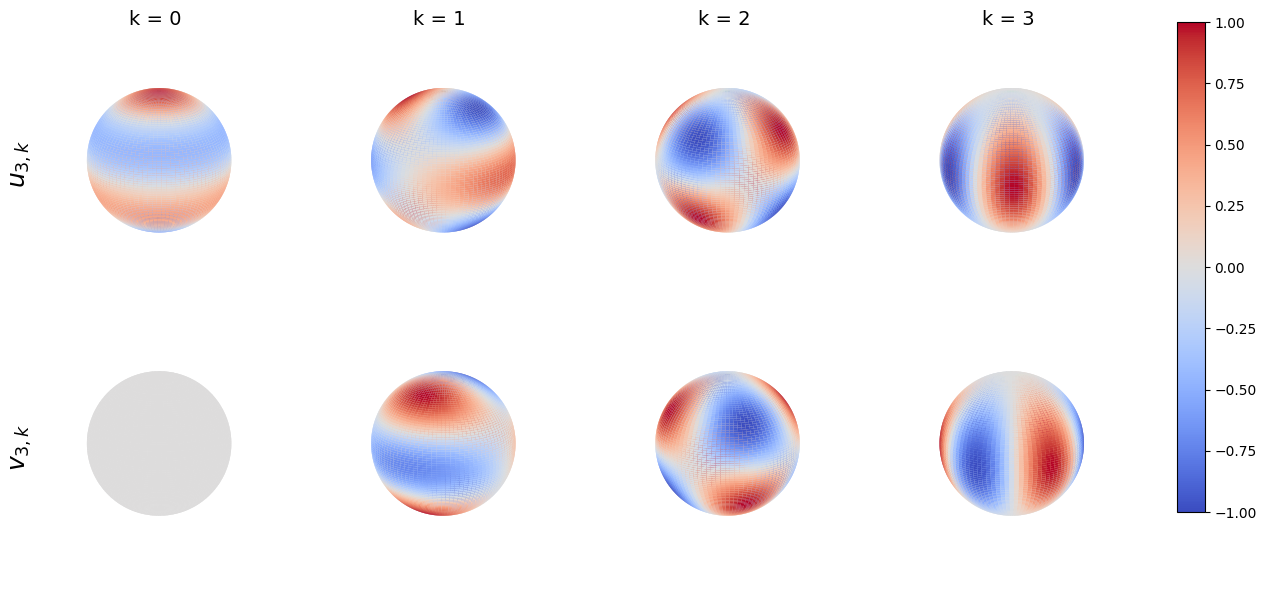}
    \caption{Visualization of the normalized angular eigenfunctions $u_{3,k}(\phi,\theta):= P_3^k(\cos\phi)\cos k\theta$ and $v_{3,k} := P_3^k(\cos \phi)\sin k\theta$ for $k=0,1,2,3$.}
    \label{fig:eigenfunctions}
\end{figure}
Choosing $q > \max\{1,2\nu\}$, we also consider the Nemytski operator $N:L^q(\Om;\bbR^N) \to L^2(\Om;\bbR^N)$ given by $ N(v)(z) := f(z,v(z))$ and the Sobolev embedding $j:\mathscr H\to L^q(\Om;\br^N)$ given by $j(u(z)) := u(z)$. Since $j$ is compact, $N$ is continuous, and $\mathscr L$ is a linear isomorphism, the family of operators
\begin{align}\label{def:operator_F}
    \mathscr F: \br \times \mathscr H \rightarrow \mathscr H, \quad \mathscr F(\alpha,u) := u - \mathscr L^{-1}(N(j(u)) + A(\alpha)j(u)),
\end{align}
defines a compact perturbation of identity for every parameter value $\alpha \in \br$. Notice also that the boundary value problem \eqref{eq:system} is equivalent to operator equation
\begin{align}\label{eq:operator_equation}
    \mathscr F(\alpha,u) = 0, 
\end{align}
in the sense that a function $u \in \mathscr H$ satisfies \eqref{eq:system} for a particular parameter value $\alpha \in \br$ if and only if $(\alpha,u) \in \br \times \mathscr H$ is a solution to \eqref{eq:operator_equation}. In what follows, we will indicate by 
$\mathscr A:= D \mathscr F(\alpha,0) : \mathscr H \rightarrow \mathscr H$ the linearization of \eqref{def:operator_F} about the trivial solution $(\alpha,0) \in \br \times \mathscr H$ which (cf. Lemma $2.2$ in \cite{Jingzhou}) is given by  
\begin{align} 
\label{def:operator_A}
\mathscr A(\alpha)u := u - \mathscr L^{-1} A(\alpha) u, \quad \alpha \in \br, \; u \in \mathscr H.
\end{align}

\section{The $G$-Isotypic Decomposition of $\mathscr H$} \label{sec:isotypic_decomp}
Assuming that a complete list of the irreducible $\Gamma$-representations $\{ \mathcal V_j \}_{j=1}^r$ is made available, we denote by $\{ \mathcal V_j^- \}_{j=1}^r$
the corresponding list of irreducible $\Gamma \times \mathbb{Z}_2$-representations, where the superscript is meant to indicate that each irreducible $\Gamma$-representation has been equipped with the antipodal $\mathbb{Z}_2$-action ($v \mapsto - v$). As a $\Gamma$-representation, $V = \mathbb{R}^N$ is also a natural $\Gamma \times \mathbb{Z}_2$-representation with a $\Gamma \times \mathbb{Z}_2$-isotypic decomposition of the form
\begin{align*}
  V = V_1 \oplus V_2 \oplus \cdots \oplus V_r,  
\end{align*}
where each $\Gamma \times \mathbb{Z}_2$-isotypic component $V_j$ is modeled on the irreducible $\Gamma \times \mathbb{Z}_2$-representation $\mathcal V_j^-$ in the sense that $V_j$ is equivalent to the direct sum of some number of copies of $\mathcal V_j^-$, i.e. $V_j \simeq \mathcal V_j^- \oplus \cdots \oplus \mathcal V_j^-$.
The exact number of irreducible $\Gamma$-representations $\mathcal V_j^-$ `contained' in the $\Gamma \times \mathbb{Z}_2$-isotypic component $V_j$ is called the \textit{$\mathcal V_j^-$-isotypic multiplicity} of $V$ and is calculated according to the ratio $m_j:= \dim V_j / \dim \mathcal{V}_j^-, \quad j \in \{1,2,\ldots,r\}$.

Notice that each of the eigenspaces \eqref{def:space_Emn} admits an induced $\Gamma \times \bz_2$ decomposition 
\[
\mathscr E_{m,n} = \bigoplus_{j=1}^r \mathscr E_{m,n,j}, \; \mathscr E_{m,n,j} := \left\{  J_m(r \sqrt{s_{m,n}})(u_{m,k}(\phi,\theta)a_k + v_{m,k}(\phi,\theta)b_k) : 0 \leq k \leq m, \; a_k,b_k \in V_j \right\},
\]
where $u_{m,k}(\phi,\theta):=  P_m^k(\cos \phi ) \cos k \theta$ and  $v_{m,k}(\phi,\theta):= P_m^k(\cos \phi ) \sin k\theta$. To simplify our computations, we introduce an additional condition on the family of matrices $A: \br \rightarrow L^\Gamma(\br^N)$:
\begin{enumerate} [label=($A_0$)] 
    \item\label{a0} For each $j \in \{1,2, \ldots, r\}$ there exists a continuous map $\mu_j: \mathbb{R} \rightarrow \mathbb{R}$ with
    \[
    A_j(\alpha) = \mu_j(\alpha) \operatorname{id}|_{V_j}, \quad A_j(\alpha) := A(\alpha)|_{V_j}: V_j \rightarrow V_j.
    \]
\end{enumerate}
On the other hand, the list of all irreducible $O(3)$-representations consists of the trivial representation $\mathcal W_0 \simeq \br$ on which $O(3)$ acts trivially and, for each $m \in \bn$, the $(2m + 1)$-dimensional representation $\mathcal W_m$, equivalent to the space of harmonic homogeneous polynomials of degree $m$ equipped with the $O(3)$-action $(\gamma p)(v) := p(\gamma^{-1} v)$ for each $p \in \mathcal W_m$, $v = (x,y,z)^T \in \br^3$ and $\gamma \in O(3)$.
At this point, every irreducible $G$-representation can be uniquely identified with an index pair $(m,j) \in \bn \cup \{0\} \times \{1,\ldots,r\}$ via the notation 
\[
\mathcal V_{m,j} := \mathcal W_m \otimes \mathcal V_j^-.
\]
Alongside Golubitsky in \cite{Golubitsky}, we notice that, for any fixed $m \in \bn$, a 
surface harmonic of the form $\sum_{k=0}^m\cos(k \theta)P_m^k(\cos(\phi))a_k + \sin(k \theta)P_m^k(\cos(\phi))b_k$ with $a_k,b_k \in \br$, when multiplied by the factor $r^m$ and expressed in Cartesian coordinates, becomes a homogeneous harmonic polynomial of degree $m$, such that one has the equivalence
\[
\mathscr E_{m,n,j} \simeq \mathcal V_{m,j}, \quad m = 0,1,\ldots, \; n = 1,2,\ldots, \; j = 1,2,\ldots,r.
\]
Consequently, $\mathscr H$ admits the $G$-isotypic decomposition 
\[
\mathscr H = \overline{\bigoplus_{m=0}^\infty \mathscr H_{m,j}}, \quad \mathscr H_m := \overline{\bigoplus_{n=1}^\infty \mathscr E_{m,n,j}},
\]
where the closure is taken in $\mathscr H$ and each isotypic component $\mathscr H_{m,j}$ is modeled on the corresponding irreducible $G$-representation $\mathcal V_{m,j}$. By Schur's Lemma, the family of $G$-equivariant linear operators $\mathscr A:\br \times \mathscr H \rightarrow \mathscr H$ respects
the $G$-isotypic decomposition in the sense that, for every $\alpha \in \br$, one has $\mathscr A(\alpha)(\mathscr H_{m,j}) \subset \mathscr H_{m,j}$.
It follows that $\mathscr A(\alpha):\mathscr H \rightarrow \mathscr H$ admits the block-matrix decomposition
\[
\mathscr A(\alpha) = \bigoplus_{m=0}^\infty \bigoplus_{j=1}^r  \mathscr A_{m,j}(\alpha), \quad \mathscr A_{m,j}(\alpha):= \mathscr A(\alpha)|_{\mathscr H_{m,j}}: \mathscr H_{m,j} \rightarrow \mathscr H_{m,j}.
\]
Under Assumption \ref{a0}, we are also guaranteed that $\mathscr A(\alpha)$ admits the following spectral decomposition 
\[
\sigma(\mathscr A(\alpha)) = \bigcup_{m=0}^\infty  \sigma(\mathscr A_{m,j}(\alpha)), \quad \sigma(\mathscr A_{m,j}(\alpha)) = \left\{ 1 - \frac{\mu_j(\alpha)}{s_{m,n}} :  n \in \bn \right\},
\]
such that every eigenvalue of $\mathscr A(\alpha)$ can be identified with an element of the index set $\Sigma := \left\{(m,n,j) : m \in \bn \cup \{0\}, \; n \in \bn, \; j \in \{1,2,\ldots,r\} \right\}$ as follows:
\begin{align}
    \mu_{m,n,j}(\alpha):= 1 - \frac{\mu_j(\alpha)} {s_{m,n}} = \frac{s_{m,n} - \mu_j(\alpha)}{s_{m,n}}, \quad (m,n,j) \in \Sigma.
\end{align}

This also shows that the action of a maximal non-radial orbit type on $\mathscr H$ can be viewed in terms of its action on particular $G$-isotypic components, where it must fix some non-trivial vector in some $\mathscr E_{m,n,j}$. Using the geometric properties of the maximal non-radial orbit types, such fixed points can be computed explicitly, and representative examples of such fixed points for a particular choice of $\Gamma$ can be seen in Figure \ref{fig:representative_solutions}.

\begin{figure}[H]
    \centering
    \includegraphics[width=\linewidth]{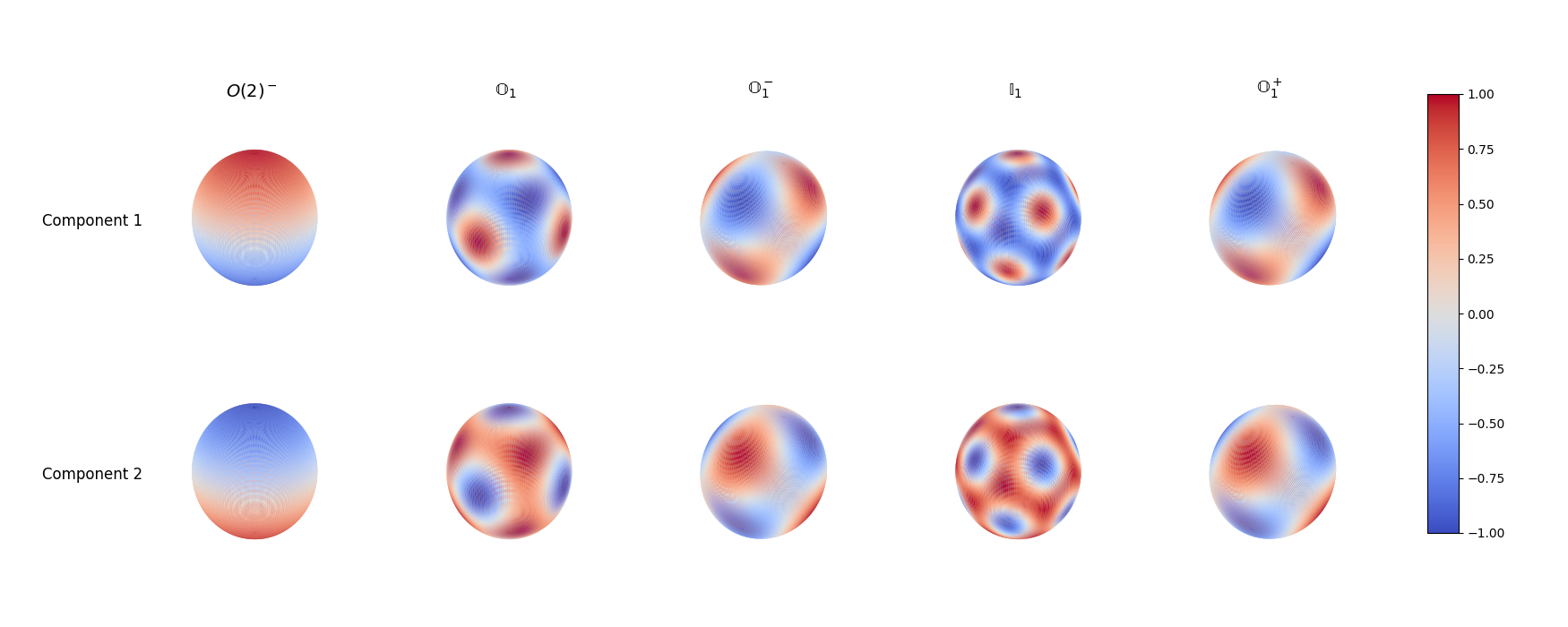}
    \caption{Representative solutions admitting 
    maximal non-radial symmetries
    for the antiphase ($j=1$) isotypic component. These patterns are constructed as linear combinations of the normalized angular eigenfunctions $u_{m,k}(\theta,\phi) :=  P_m^k(\cos \phi)\cos k\theta$: $O(2)^-$ fixes $u_{1,0}$, $\mathbb O_1$ fixes $ u_{4,0} + \sqrt{{10}/{7}} u_{4,4}$, $\mathbb O_1^-$ and $\mathbb O_1^+$ fix $u_{3,0} + \sqrt{{8}/{5}} u_{3,3}$, and $\mathbb I_1$ fixes $ u_{6,0} + \sqrt{{28}/{11}} u_{6,5}$. Note that $\mathbb O_1^-$ and $\mathbb O_1^+$ fix the same subspace through different group actions (on odd and even Fourier modes, respectively).
    For the in-phase case ($j=0$), solutions with the same spatial symmetries exist where the patterns on both spheres are identical.}
    \label{fig:representative_solutions}
\end{figure}

\begin{remark} \label{rm:critical_values} \rm
A trivial solution $(\alpha_0,0) \in \br \times \br \times \mathscr H$ belongs to the critical set of \eqref{eq:operator_equation} if and only if $0 \in \sigma(\mathscr A(\alpha_0))$. In other words, $(\alpha_0,0) \in \Lambda$ if and only if 
there exists an index triple $(m,n,j) \in \Sigma$
for which $\mu_j(\alpha_0) = s_{m,n}$.
\end{remark}
\section{Existence Results For \eqref{eq:system}}\label{sec:existence_result}
Given any parameter value $\alpha \in \br$ for which $\mathscr A(\alpha):\mathscr H \rightarrow \mathscr H$ is an isomorphism, the product property of the $G$-equivariant Leray-Schauder degree (cf. Appendix \ref{sec:appendix}) permits us to express the degree $\gdeg(\mathscr A(\alpha), B(\mathscr H))$ in terms of a Burnside ring product of the $G$-equivariant Leray-Schauder degrees of the Fourier mode restrictions $\mathscr A_{m,j}(\alpha): \mathscr H_{m,j} \rightarrow \mathscr H_{m,j}$ of the $G$-equivariant linear isomorphism $\mathscr A: \mathscr H \rightarrow \mathscr H$ to the open unit balls $B(\mathscr  H_{m,j}):= \{ u \in \mathscr H_{m,j} \; : \; \| u \|_{\mathscr H} < 1 \}$ as follows
\begin{equation}\label{eq:gdegA_product_property_decomp}
  \gdeg(\mathscr A(\alpha), B(\mathscr H)) = \prod\limits_{m=0}^\infty \prod\limits_{j=1}^r \gdeg( \mathscr A_{m,j}(\alpha), B(\mathscr H_{m,j})).  
\end{equation}
On the other hand, each $\gdeg( \mathscr A_{m,j}, B(\mathscr H_{m,j}))$ is fully specified by the corresponding negative spectra $\sigma_-(\mathscr A_{m,j}(\alpha)) := \{ \mu_{m,n,j}(\alpha) < 0 : n \in \bn, \; j \in \{1,2,\ldots,r\}\}$ according to the formula
\begin{equation}\label{eq:negative_spectrum_gdegA}
\gdeg(\mathscr A_{m,j}, B( \mathscr H_{m,j})) = \prod_{\mu_{m,n,j}(\alpha) \in \sigma_-(\mathscr A_{m,j}(\alpha))} (\deg_{\mathcal V_{m,j}})^{m_j},
\end{equation}
where $\deg_{\mathcal V_{m,j}} \in A(G)$ is the basic degree associated with the irreducible $G$-representation $\mathcal V_{m,j}$ (cf. Appendix \ref{sec:appendix}) and $(G) \in A(G)$ is the unit element of the Burnside ring. 
Putting together \eqref{eq:gdegA_product_property_decomp} and \eqref{eq:negative_spectrum_gdegA} and adopting the notation 
\begin{align}
\label{def:index_set_Sigma}
   \Sigma(\alpha) := \left\{ (m,n,j) \in \Sigma :  \mu_{m,n,j}(\alpha) < 0 \right\},
\end{align}
to account for the negative spectrum of $\mathscr A(\alpha)$, one has
\begin{align} \label{eq:computational_formula_gdegA}
    \gdeg(\mathscr A(\alpha), B(\mathscr H)) \; = \prod\limits_{(m,n,j) \in \Sigma(\alpha)} (\deg_{\mathcal V_{m,j}})^{m_j}.
\end{align}
For any parameter value $\alpha \in \br$ and maximal orbit type $(H) \in \Phi_0(G; \mathscr H\setminus\{0\})$, we put
\begin{align}\label{def:crossing_number}
        \mathfrak m^H(\alpha) := \left|\{(m,n,j) \in \Sigma(\alpha) :  2 \nmid  \dim\mathcal V_{m,j}^H, \; 2 \nmid m_j \}\right|.
\end{align}
We now have all the necessary tools to describe a practical rule for the nontriviality of any maximal orbit type in the Burnside ring product \eqref{eq:computational_formula_gdegA}.
\begin{theorem}
For any parameter value $\alpha \in \br$ at which $\mathscr A(\alpha)$ is an isomorphism and for every maximal orbit type $(H) \in \Phi_0(G; \mathscr H\setminus\{0\})$, one has
\begin{align}\label{eq:H_coeffrule_regularpoint}
\operatorname{coeff}^{H}\left( \gdeg(\mathscr A(\alpha) , B(\mathscr H)) \right) = \begin{cases}
    n_H & \text{ if } \mathfrak m^H(\alpha) \text{ is  odd}; \\
    0 & \text{ if } \mathfrak m^H(\alpha) \text{ is  even},
\end{cases}
\end{align}
where $\mathfrak m^H(\alpha)$ is given by \eqref{def:crossing_number} and
\begin{align*}
    n_H := \begin{cases}
        1 & \text{ if } |W(H)| = 2; \\
        2 &\text{ if } |W(H)| = 1.
    \end{cases}
\end{align*} 
\end{theorem}
\begin{proof}
For any irreducible $G$-representation $\mathcal V_{m,j}$ and maximal orbit type $(H) \in \Phi_0(G)$, the recurrence formula for the $G$-equivariant degree (cf. Appendix \ref{sec:appendix}) implies
\[
\deg_{\mathcal V_{m,j}} = (G) - x_{m,j}(H) + \bm a,
\]
where $\bm a \in A(G)$ satisfies $\operatorname{coeff}^H(\bm a) = 0$ and the coefficient $x_{m,j}:= -\operatorname{coeff}^H(\deg_{\mathcal V_{m,j}})$ is given by
\[
x_{m,j} = \frac{(-1)^{\dim \mathcal V_{m,j}^H}-1}{|W(H)|}.
\]
Clearly, for any $m,m' \in \bn$ and $j,j ' \in \{1,\ldots,r\}$ with $\dim \mathcal V_{m,j}^H$ and $\dim \mathcal W_{m',j'}^H$ both even, one has $\operatorname{coeff}^H(\deg_{\mathcal V_{m,j}} \cdot \deg_{\mathcal V_{m',j'}}) = 0$. Let's consider the behavior of $(H)$ in the case that both $\dim \mathcal V_{m,j}^H$ and $\dim \mathcal W_{m',j'}^H$ are odd:
\begin{align*}
\deg_{\mathcal V_{m,j}} \cdot \deg_{\mathcal V_{m',j'}} &= ((G) - n_H(H) + \bm a_{m,j}) \cdot ((G) - n_H(H) + \bm a_{m',j'}) \\
&= (G) - 2n_H(H) + n_H^2(H) \cdot (H) + \bm b,
\end{align*}
where by the recurrence formula \eqref{eq:RF-0}, one has
\begin{align*}
(H) \cdot (H) = |W(H)|(H) + \bm c,    
\end{align*}
where $\bm c \in A(G)$ satisfies $\operatorname{coeff}^H(\bm a) = 0$, such that $\operatorname{coeff}^H(\deg_{\mathcal V_{m,j}} \cdot \deg_{\mathcal V_{m',j'}}) = 0$. Therefore, only in the case that $\dim \mathcal V_{m,j}^H$ and $\dim \mathcal V_{m',j'}^H$ admit different parities does one have \\ $\operatorname{coeff}^H(\deg_{\mathcal V_{m,j}} \cdot \deg_{\mathcal V_{m',j'}}) = n_H$.
\end{proof}

\section{Local and Global Bifurcation in 
\eqref{eq:system}} \label{sec:bifurcation_results}
The set of all solutions to the operator equation \eqref{eq:operator_equation} can be divided into the set of trivial solutions $M:= \{(\alpha,0) \in \br \times \mathscr H \}$ and the set of nontrivial solutions $\mathscr S:= \{(\alpha,u) \in \br \times \mathscr H : \mathscr  F(\alpha,u) = 0, \; u \not\equiv 0\}$.
Given any orbit type $(H) \in \Phi_0(G)$, we can always consider the $H$-fixed-point set 
\[
\mathscr S^H := \{(\alpha,u) \in \mathscr S: G_u \leq H \},
\]
consisting of all nontrivial solutions to \eqref{eq:operator_equation} with {\bf symmetries at least $(H)$}, i.e. $(\alpha,u) \in \mathscr S^H$ if and only if $\mathscr  F(\alpha,u) = 0$, $u \in \mathscr H \setminus \{0\}$ and 
\[
h u(r,\theta) = u(r,\theta), \; \text{ for all } h \in H \text{ and } (r,\theta) \in D.
\]
\subsection{The Local Bifurcation Invariant and Krasnosel'skii's Theorem}\label{sec:local-bif-inv}
Formulation of a Krasnosel'skii type local bifurcation result for the boundary value problem \eqref{eq:system} requires additional notation and terminology (for more details, see \cite{AED, book-new}).
\vs
We begin by clarifying what constitutes bifurcation for equation \eqref{eq:operator_equation}. A trivial solution $(\alpha_0,0) \in M$ is called a \textbf{bifurcation point} if every open neighborhood of $(\alpha_0,0)$ contains nontrivial solutions from $\mathscr S$. A necessary condition for bifurcation is that the linear operator $\mathscr A(\alpha_0):\mathscr H \rightarrow \mathscr H$ fails to be an isomorphism, which motivates a classification of the trivial solution branch: we say that $(\alpha_0,0) \in M$ is a \textbf{regular point} when $\mathscr A(\alpha_0)$ is an isomorphism, and a \textbf{critical point} otherwise. An \textbf{isolated critical point} is one that admits a deleted $\epsilon$-neighborhood $0< \vert \alpha - \alpha_0 \vert < \epsilon$ where all nearby parameter values correspond to regular points. The collection of all critical points forms the \textbf{critical set}
\begin{align}\label{eq:critical}
\Lambda:=\left\{(\alpha,0) \in \br \times \mathscr H: \text{ $\mathscr A(\alpha):\mathscr H \rightarrow \mathscr H$ is not an isomorphism}\right\}.
\end{align}
Not every critical point necessarily gives rise to bifurcation. We distinguish those that do by calling $(\alpha_0,0) \in M$ a \textbf{branching point} if there exists a nontrivial continuum $K \subset \overline{\mathscr S}$ with $K \cap M = { (\alpha_0,0) }$. Any maximal connected set $\mathscr C \subset \overline{\mathscr S}$ containing such a branching point is termed a \textbf{branch} of nontrivial solutions bifurcating from $(\alpha_0,0)$.
\vs
While classical Krasnosel'skii bifurcation theory only concerns the existence of branches of non-trivial solutions emerging from the trivial solution branch, the equivariant version we employ in this paper also characterizes the symmetric properties of solutions on these branches. For any subgroup $H \leq G$, we denote by $\mathscr S^H$ the $H$-fixed point space of nontrivial solutions and we say that a branch $\mathscr C$ admits \textbf{symmetries at least} $(H)$ when $\mathscr C \cap \overline{\mathscr S^H} \neq \emptyset$.
\vs
With these preliminaries out of the way, let $(\alpha_0,0) \in \Lambda$ be an isolated critical point with a deleted $\epsilon$-neighborhood $\{ \alpha \in \mathbb{R} : 0 < | \alpha - \alpha_0 | < \epsilon \}$ on which $\mathscr  A(\alpha): \mathscr H \rightarrow \mathscr H$ is an isomorphism and choose $\alpha^\pm_0 \in (\alpha_0 - \ve, \alpha_0 + \ve)$ with $\alpha^-_0 \leq \alpha_0 \leq \alpha^+_0$. Since the two operators $\mathscr  A(\alpha^\pm): \mathscr H \rightarrow \mathscr H$ are non-singular, there exists a number $\delta >0$ sufficiently small such that, adopting the notations $\mathscr  F_{\pm}(u) := \mathscr  F(\alpha^\pm_0, u)$, $B_{\delta}(\mathscr H) := \{u\in \mathscr H: \|u\|_\mathscr H< \delta\}$, one has
\begin{enumerate}
    \item[$(\rm i)$] $\mathscr  F_{\pm}^{-1}(0) \cap \partial B_{\delta}(\mathscr H) = \emptyset$,
    \item[$(\rm ii)$] $\mathscr  F_{\pm}$ are $B_{\delta}(\mathscr H)$-admissibly $G$-homotopic to $\mathscr  A(\alpha^\pm)$, respectively. 
\end{enumerate}
It follows, from the homotopy property of the $G$-equivariant Leray-Schauder degree (cf. Appendix \ref{sec:appendix}), that $(\mathscr  F_{\pm},B_{\delta})$ are admissible $G$-pairs in $\mathscr H$ and also that $\gdeg(\mathscr  F_{\pm},B_{\delta}) = \gdeg(\mathscr  A(\alpha^\pm_0), B(\mathscr H))$, where $B(\mathscr H)$ is the open unit ball in $\mathscr H$. We 
call the Burnside ring element
\begin{align} \label{def:local_bifurcation_invariant}
 \omega_{G}(\alpha_0):=\gdeg(\mathscr  A(\alpha^-_0), B(\mathscr H))-\gdeg(\mathscr  A(\alpha^+_0), B(\mathscr H)), 
\end{align}
the {\bf local bifurcation invariant} at $(\alpha_0,0)$. The reader is referred to \cite{book-new, AED} for proof that the invariant \eqref{def:local_bifurcation_invariant} does not depend on the choice of $\alpha^\pm_0 \in \mathbb{R}$ or radius $\delta >0$, and also for the proof of the following local bifurcation result.
 \vs
 \begin{theorem}\em 
 {\bf(M.A. Krasnosel'skii-Type Local Bifurcation)}\label{th:Kras}
 Let $\mathscr  F: \mathbb{R} \times \mathscr H \rightarrow \mathscr H$ be a completely continuous $G$-equivariant field 
 with a linearization $D \mathscr  F(0) : \br \times \mathscr H \rightarrow \mathscr H$ admitting 
 an isolated critical point $(\alpha_0,0)$. If $\omega_{G}(\alpha_0) \neq 0$, then there exists a branch of nontrivial solutions $\mathscr C$ to system \eqref{eq:operator_equation} with branching point $(\alpha_0,0)$. Moreover if $(H) \in \Phi_0(G)$ is an orbit type with $\operatorname{coeff}^{H}
         (\omega_{G}(\alpha_0)) \neq 0$,
then there exists a branch of nontrivial solutions bifurcating from $(\alpha_0,0)$ with symmetries at least $(H)$.
\end{theorem}
At this point we can present our main local bifurcation result.
\begin{theorem}\label{thm:main_local_bif}
    Let $(\alpha_0,0) \in \Lambda$ be an isolated critical point. If there exists a maximal non-radial orbit type $(H) \in \Phi_0(G)$ for which the parities of $\mathfrak m^H(\alpha_0^\pm)$ disagree, then the system \eqref{eq:system} admits a branch of non-radial solutions emerging from the trivial solution at $(\alpha_0,0)$ with symmetries at least $(H)$.
\end{theorem}
\begin{proof}
    Since $\mathscr A(\alpha_0^\pm): \mathscr H \rightarrow \mathscr H$ are isomorphisms, the coefficient standing next to $(H) \in \Phi_0(G; \mathscr H \setminus \{0\})$ at each of the basic degrees $\gdeg(\mathscr  A(\alpha^\pm_0), B(\mathscr H))$ is determined by the parity of $\mathfrak m^H(\alpha_0^\pm)$ according to the rule \eqref{eq:H_coeffrule_regularpoint}. Therefore, only in the case that the parities of $\mathfrak m^H(\alpha_0^\pm)$ disagree is $\operatorname{coeff}^H(\omega_G(\alpha_0))$ nontrivial. 
\end{proof}

\subsection{The Rabinowitz Alternative}\label{sec:rab_alt}
In order to employ the Leray-Schauder $G$-equivariant degree to describe the global properties of branches of nontrivial solutions bifurcating from the critical points of the equation \eqref{eq:operator_equation}, we need to make an additional assumption: 
\begin{enumerate}[label=($B$)]
\item\label{b} The critical set $\Lambda \subset M$ (given by \eqref{eq:critical}) is discrete.
\end{enumerate}
Notice $(\rm i)$ that the local bifurcation invariant $\omega_{G}(\alpha_0)$ at any critical point $(\alpha_0,0) \in \Lambda$ is well-defined under assumption \ref{b} and $(\rm ii)$ that if $\mathcal U \subset \mathbb{R} \times \mathscr H$ is an open bounded $G$-invariant set, then its intersection with the critical set is finite. These observations are important to the statement of the following global bifurcation result, the proof of which can be found in \cite{book-new, AED}:
\vs
\begin{theorem}\label{th:Rabinowitz-alt}{\bf (The Rabinowitz Alternative)} \rm
Let $\mathcal U \subset \mathbb{R} \times \mathscr H$  be an open bounded $G$-invariant set with $\partial \mathcal U \cap \Lambda = \emptyset$. If $\mathcal C$ is a branch of nontrivial solutions to \eqref{eq:system} bifurcating from the critical point $(\alpha_0,0) \in \mathcal U \cap \Lambda$, then one has the following alternative:
\begin{enumerate}[label=$(\alph*)$]
\item \label{alt_a}  either $\mathcal C \cap \partial \mathcal U \neq \emptyset$;
    \item \label{alt_b} or there exists a finite set
    \begin{align*}
        \mathcal C \cap \Lambda = \{ (\alpha_0,0),(\alpha_1,0), \ldots, (\alpha_n,0) \},
    \end{align*}
    satisfying the following relation
    \begin{align*}
\sum\limits_{k=0}^n \omega_{G}(\alpha_k) = 0.
    \end{align*}
\end{enumerate} 
\end{theorem}
At this point we can prove our main global bifurcation result, Theorem \ref{thm:main_global_bif}:
\begin{proof}[Proof of Theorem \ref{thm:main_global_bif}]
Due to the continuous dependence of the spectrum $\sigma(\mathscr A(\alpha))$ on the bifurcation parameter, one has for any two adjacent critical points $(\alpha_k,0), (\alpha_{k+1},0) \in \Lambda$ (here $0 \leq k < N $) with corresponding regular neighborhoods $[\alpha_k^-,\alpha_k^+] \setminus \{\alpha_k\}$ and $[\alpha_{k+1}^-,\alpha_{k+1}^+] \setminus \{\alpha_{k+1}\}$, the coincidence $\mathfrak m^H(\alpha_k^+) = \mathfrak m^H(\alpha_{k+1}^-)$. Therefore, the coefficient standing next to $(H) \in \Phi_0(G)$ in the sum of local bifurcation invariants $\sum_{k=1}^n \omega_G(\alpha_k)$ is determined according to the telescoping relation
\begin{align*}
\operatorname{coeff}^H( \sum_{k=1}^n &\omega_G(\alpha_k) ) = \sum_{k=1}^n\operatorname{coeff}^H\bigg(\gdeg(\mathscr  A(\alpha^-_k), B(\mathscr H))\bigg) - \operatorname{coeff}^H\bigg(\gdeg(\mathscr  A(\alpha^+_k), B(\mathscr H))\bigg) \\
&= \operatorname{coeff}^H\bigg(\gdeg(\mathscr  A(\alpha^-_1), B(\mathscr H))\bigg) - \operatorname{coeff}^H\bigg(\gdeg(\mathscr  A(\alpha^+_n))\bigg),
\end{align*}
at which point the result follows with reasoning similar to that used in the proof of Theorem \ref{thm:main_local_bif}.
\end{proof}
\section{Motivating Example: A System of Two Spherical Oscillators With $\Gamma = S_2$ Coupling} \label{sec:example}
We now illustrate the general theory in the simplest nontrivial case by setting $N=2$ with $\Gamma=S_2$ (such that the symmetry group becomes $G=O(3)\times S_2\times\bz_2$)  and modeling a pair of coupled oscillators on the unit ball via the equations
\begin{equation}\label{eq:example-system}
\begin{cases}
-\Delta u = f(x,u) + A(\alpha)u, \qquad u(x)\in\br^2,\\
u|_{\partial\Omega}=0,
\end{cases}
\end{equation}
where $A:\br\to L^{S_2}(\br^2)$ is a family of $S_2$–equivariant (symmetric) matrices and $f$ satisfies Assumptions \ref{a1}---\ref{a3}. 
\vs
The cyclic group $S_2 \simeq \bz_2$ admits the trivial irreducible representation $\mathcal V_0 := \langle(1,1)\rangle \simeq \br$ and the sign representation $\mathcal V_1=\langle(1,-1)\rangle \simeq \br$. Notice the $S_2$-isotypic decomposition of $\br^2$
\[
\br^2 \simeq \mathcal V_0 \oplus \mathcal V_1,
\]
has simple $S_2$–isotypic components for both indices $j \in \{0,1\}$, i.e. $m_0=m_1=1$. By Schur's lemma the family of $S_2$-equivariant matrices $A(\alpha)$ acts diagonally on the $S_2$-isotypic decompostion via $A(\alpha)|_{\mathcal V_j}=\mu_j(\alpha) \id,\; j\in\{0,1\}$, obviating the need for Assumption \ref{a0}. Moreover, for each index triple $(m,n,j) \in
\Sigma$ (in this setting $\Sigma :=
\{0,1,\ldots\} \times \{1,2,\ldots \} \times \{0,1\}$), there is an associated irreducible $G$–representation $\mathcal V_{m,j}=\mathcal W_m\otimes\mathcal V_j^-$, and the linearization $\mathscr A(\alpha)=D\mathscr F(\alpha,0)$ restricts to  the corresponding $G$-isotypic component $\mathscr H_{m,j}$ with spectrum
\[
\sigma(\mathscr A_{m,j}(\alpha))=\left\{\,1-\frac{\mu_j(\alpha)}{s_{m,n}}:n\in\bn\right\}.
\]
The critical set consists of all trivial solutions $(\alpha_0,0)$ satisfying $\mu_j(\alpha_0)=s_{m,n}$ for some $(m,n,j)$. For any regular parameter value $\alpha \in \br$, we define the index set $\Sigma(\alpha):=\{(m,n,j) \in \Sigma:\mu_j(\alpha)>s_{m,n}\}$
such that
\[
\gdeg(\mathscr A(\alpha),B(\mathscr H))=\prod_{(m,n,j)\in\Sigma(\alpha)}\deg_{\mathcal V_{m,j}}.
\]
The coefficient of a maximal orbit type $H$ in this product is fully determined by the parity of
\[
\mathfrak m^H(\alpha):=\bigg|\{(m,n,j)\in\Sigma(\alpha):2\nmid\dim\mathcal V_{m,j}^H\}\bigg|.
\]
Notice that the orbit types $H$ for which $\dim\mathcal V_{m,j}^H$ is odd (see Tables \ref{table:max-orbit-types-abbrev} and \ref{table:max-orbit-types-sequential}) are associated with a \emph{unique} isotypic index $j(H)\in\{0,1\}$. Writing
$\mathcal N_m(x):=\big|\{n\ge1:s_{m,n}<x\}\big|$,
we obtain
\begin{equation}\label{eq:mfH-general}
\mathfrak m^H(\alpha)=\sum_{m\in\mathcal M_H}\mathcal N_m(\mu_{j(H)}(\alpha)),
\end{equation}
where $\mathcal M_H$ is the set of $m$ for which $H$ appears in Table~\ref{table:max-orbit-types-sequential} for its associated index $j(H)$.
\vs
\noindent\textbf{Specialization to a concrete model.}
Consider the case of
\begin{align*}
A(\alpha)=\zeta(\alpha)(a\,I+b\,P),\qquad P=\begin{pmatrix}0&1\\1&0\end{pmatrix},
\end{align*}
with constants $a,b\in\br$ and a continuous profile given by the sigmoid function 
(see Figure \ref{Fig:sigmoid}). Then
\[
\mu_0(\alpha)=\zeta(\alpha)(a+b),\qquad \mu_1(\alpha)=\zeta(\alpha)(a-b),
\]
and \eqref{eq:mfH-general} becomes
\begin{equation}\label{eq:mfH-model}
\mathfrak m^H(\alpha)=\sum_{m\in\mathcal M_H}\mathcal N_m\big(\zeta(\alpha)(a+(-1)^{j(H)}b)\big).
\end{equation}
We are now positioned to translate the abstract conditions of Theorems \ref{thm:main_local_bif} and \ref{thm:main_global_bif} into concrete criteria based on the model parameters $a$ and $b$.
\vs
Local bifurcation of non-radial solutions with maximal symmetry $(H)$ occurs at a critical parameter value $\alpha_0$ if the parity of $\mathfrak m^H(\alpha)$ changes as $\alpha$ passes through $\alpha_0$. From Equation \eqref{eq:mfH-model}, this happens precisely when $\mu_{j(H)}(\alpha_0) = \zeta(\alpha_0)(a+(-1)^{j(H)}b)$ crosses a Bessel root $s_{m,n}$ for some $m \in \mathcal{M}_H$.
\vs
For global bifurcation, we examine the parity of $\mathfrak m^H(\alpha)$ at the extremes of the parameter range. As $\alpha \to -\infty$, we have $\zeta(\alpha) \to 0$, which implies $\mu_j(\alpha) \to 0$ for $j=0,1$. Assuming $a \pm b > 0$, no Bessel roots are crossed, so $\mathfrak m^H(\alpha) \to 0$, which is an even number. As $\alpha \to +\infty$, we have $\zeta(\alpha) \to 1$, which implies $\mu_0(\alpha) \to a+b$ and $\mu_1(\alpha) \to a-b$. In this limit, the parity of $\mathfrak m^H(\alpha)$ is determined by the parity of $\sum_{m\in\mathcal M_H}\mathcal N_m(a+(-1)^{j(H)}b)$. Applying Theorem \ref{thm:main_global_bif}, we need the parity of $\mathfrak m^H(\alpha)$ to differ at the start and end of the parameter range. 
\vs
We now have a simple algebraic check on the parameters $a,b \in \br$ to guarantee the existence and unboundedness of non-radial solution patterns,
which we can use to prove Theorem \ref{thm:exampleo2}:
\begin{proof}[Proof of Theorem \ref{thm:exampleo2}]
In order to find conditions on $a$ and $b$ that guarantee an unbounded branch of solutions with symmetry $(H)$ for each $H \in \{ O(2)^-_j, \mathbb O_j, \mathbb O^-_j, \mathbb O^+_j, \mathbb I_j : j = 0,1 \}$, we simply consult Table \ref{table:max-orbit-types-sequential} to identify the sets $\mathcal M_H$ and apply Theorem \ref{thm:main_global_bif}.
\end{proof}
\vs

\newpage
\appendix
\section{Twisted and Amalgamated Notation}\label{app:amalgamated_notation}
Up to their conjugacy classes, the nontrivial subgroups of $SO(3)$ consist of the planar subgroups $O(2), SO(2), \bz_n, D_n$ (we identify $D_2$ with the Klein group $V_4$) for $n \in \bn$ which leave the $(x,y)$-plane invariant in $\br^3$ and the exceptional subgroups $A_4, S_4$ and $A_5$ which are the symmetry groups associated with the regular tetrahedron, octahedron and icosohedron, respectively. Since the orthogonal group $O(3)$ is isomorphic to the product group $SO(3) \times \bz_2$, every closed subgroup of $O(3)$ can be classified (again, up to conjugacy) as a product subgroup of the form $H^p := H \times \bz_2$ for some $H \leq SO(3)$ or a twisted subgroup, which can be uniquely identified by a subgroup $H \leq SO(3)$ and a homomorphism $\varphi: H \rightarrow \bz_2$ using \emph{twisted notation}, as follows
\[
H^\varphi := \{(g,z) \in SO(3) \times \bz_2 : \varphi(g) = z\}.
\]
Adopting the auxiliary notations $(\rm i)$ $z: D_n \rightarrow \bz_2$ to indicate the epimorphism satisfying $\ker(z) = \bz_n$, $(\rm ii)$ `$d$' to indicate either of the epimorphisms $d: D_{2n} \rightarrow \bz_2$ with $\ker(d) = D_n$ or $d: \bz_{2n} \rightarrow \bz_2$ with $\ker(d) = \bz_n$ and $(\rm iii)$ `$-$' to indicate any of the epimorphisms $-: S_4\rightarrow \bz_2$ with $\ker(-)= A_4$, $-:  O(2) \rightarrow \bz_2$ with $\ker(-)= SO(2)$, or $-:  O(3) \rightarrow \bz_2$ with $\ker(-)= SO(3)$, we can describe the full list of subgroups $H \leq O(3)$ admitting finite Weyl groups $W(H) = N(H)/H$ (cf. Table \ref{table:twisted_subgroups}).
\vs
\begin{table}[htbp] 
\centering 
\caption{Twisted subgroups of O(3)}
\label{table:twisted_subgroups}  
    \begin{tabular}{lllllll} 
    \hline
    $H^\sigma$ & $H$ & $\varphi(H)$ & Ker $\varphi$ & $N(H)$ & $W(H)$ & Notes \\
    \hline
    $O(2)$ & $O(2)$ & $\mathbb{Z}_1$ & $O(2)$ & $O(2) \times \mathbb{Z}_2$ & $\mathbb{Z}_2$ & \\
    $SO(2)$ & $SO(2)$ & $\mathbb{Z}_1$ & $SO(2)$ & $O(2) \times \mathbb{Z}_2$ & $\mathbb{Z}_2 \times \mathbb{Z}_2$ & \\
    $D_n$ & $D_n$ & $\mathbb{Z}_1$ & $D_n$ & $D_{2n} \times \mathbb{Z}_2$ & $\mathbb{Z}_2 \times \mathbb{Z}_2$ & $n > 2$ \\
    $V_4$ & $V_4$ & $\mathbb{Z}_1$ & $V_4$ & $S_4 \times \mathbb{Z}_2$ & $S_3 \times \mathbb{Z}_2$ & \\
    $S_4$ & $S_4$ & $\mathbb{Z}_1$ & $S_4$ & $S_4 \times \mathbb{Z}_2$ & $\mathbb{Z}_2$ & \\
    $A_5$ & $A_5$ & $\mathbb{Z}_1$ & $A_5$ & $A_5 \times \mathbb{Z}_2$ & $\mathbb{Z}_2$ & \\
    $A_4$ & $A_4$ & $\mathbb{Z}_1$ & $A_4$ & $S_4 \times \mathbb{Z}_2$ & $\mathbb{Z}_2 \times \mathbb{Z}_2$ & \\
    $O(3)^-$ & $O(3)$ & $\mathbb{Z}_2$ & $SO(3)$ & $O(3) \times \mathbb{Z}_2$ & $\mathbb{Z}_2$ & \\
    $O(2)^-$ & $O(2)$ & $\mathbb{Z}_2$ & $SO(2)$ & $O(2) \times \mathbb{Z}_2$ & $\mathbb{Z}_2$ & \\
    $S_4^-$ & $S_4$ & $\mathbb{Z}_2$ & $A_4$ & $S_4 \times \mathbb{Z}_2$ & $\mathbb{Z}_2$ & \\
    $D_k^z$ & $D_k$ & $\mathbb{Z}_2$ & $\mathbb{Z}_k$ & $D_{2k} \times \mathbb{Z}_2$ & $\mathbb{Z}_2 \times \mathbb{Z}_2$ & $k > 3$ \\
    $D_{2n}^d$ & $D_{2n}$ & $\mathbb{Z}_2$ & $D_k$ & $D_{2n} \times \mathbb{Z}_2$ & $\mathbb{Z}_2$ & $n>2$ \\
    $V_4^-$ & $V_4$ & $\mathbb{Z}_2$ & $\mathbb{Z}_2$ & $D_4 \times \mathbb{Z}_2$ & $\mathbb{Z}_2 \times \mathbb{Z}_2$ & \\
    \hline
    \end{tabular}
    \\[1ex] %
\end{table}
In turn, adopting the convention of {\bf amalgamated notation}, a shorthand for identification of the subgroups in a product group first considered by Balanov et al. in \cite{AED}, we can identify each subgroup in $O(3) \times \Gamma \times \bz_2$ with a quintuple $(K_O,K_{\Gamma}, \psi_O,\psi_{\Gamma},L)$ consisting of two subgroups $K_O \leq O(3)$, $K_\Gamma \leq \Gamma \times \mathbb{Z}_2$, a group $L$ and a pair of epimorphisms $\psi_O: K_O \rightarrow L$, $\psi_\Gamma: K_\Gamma \rightarrow L$, as follows
\begin{align} \label{def:amalgamated_notation}
K_O {}^{\psi_O}\times^{\psi_\Gamma}_{L} K_\Gamma := \{ (x,y) \in K_O \times K_\Gamma \; : \; \psi_O(x) = \psi_\Gamma(y) \}.
\end{align}
In particular, if we put $Z_O := \ker \psi_O$, $Z_\Gamma := \ker \psi_\Gamma$ and choose $L$ such that $L \simeq K_O/Z_O \simeq K_\Gamma/Z_\Gamma$, then the amalgamated subgroup \eqref{def:amalgamated_notation} can be identified, up to its conjugacy class, with the quadruple $(K_O,K_\Gamma,Z_O,Z_\Gamma)$ using amalgamated notation, as follows
\begin{align} \label{def:compact_amalgamated_notation}
(K_O {}^{Z_O}\times^{Z_\Gamma} K_\Gamma) := 
(K_O {}^{\psi_O}\times^{\psi_\Gamma}_{L} K_\Gamma).
\end{align}
Notice that, a closed subgroup $H \leq O(3) \times \Gamma \times \bz_2$ with the amalgamated decomposition \eqref{def:amalgamated_notation} admits a finite Weyl group in $G$ if and only if $K_O$ admits a finite Weyl group in $O(3)$. 

\section{The Isotropy Lattice $\Phi_0(O(3) \times S_2 \times \mathbb Z_2)$.} \label{app:max_orbit_types}
In this appendix, we consider the case of $\Gamma = S_2$ and derive the maximal non-radial isotropy subgroups for the group  $G:= O(3) \times S_2 \times \mathbb Z_2$ acting on the solution space $\mathscr H$. The argument demonstrates how these symmetries can be systematically constructed from the known maximal isotropy subgroups of $O(3)$ acting on spaces of spherical harmonics.
\vs
To begin, we establish the structure of the irreducible $G$-representations on which our analysis is based. Writing $O(3)\cong SO(3)\times \mathbb Z_2$, the group $G$ can be identified with the product group $SO(3) \times \bz_2 \times S_2 \times \bz_2$. An irreducible $G$-representation $\mathcal V_{m,j}$ (as defined in Section \ref{sec:isotypic_decomp}) is constructed as a tensor product of the irreducible representations of these component groups. The action of $O(3)$ on the space of spherical harmonics of degree $m$, denoted $\mathcal W_m$, is determined by the action of the central inversion element $-I \in O(3)$. For a function $u(x) \in \mathcal W_m$ composed of such harmonics, this action is given by
\[
-I \cdot u(x) := u(-x) = (-1)^m u(x).
\]
Thus, the irreducible $O(3)$-representation is of type $\mathcal U_m \otimes \mathcal Z_0$ if $m$ is even and is of type $\mathcal U_m \otimes \mathcal Z_1$ if $m$ is odd, where $\mathcal U_m$ is the standard $(2m+1)$-dimensional $SO(3)$-representation and $\mathcal Z_0, \mathcal Z_1 \simeq \br$ are the trivial and sign representations of $\bz_2$, respectively.
\vs
It follows that the irreducible $G$-representation $\mathcal V_{m,j}$, on which the $G$-isotypic component $\mathscr H_{m,j}$ is modeled on, admits the tensor decomposition
\begin{align*}
    \mathcal V_{m,j} = 
 \begin{cases}
\mathcal V_m \otimes \mathcal Z_0 \otimes \mathcal Z_j \otimes \mathcal Z_1 & \text{ if $m$ is odd}; \\
\mathcal V_m \otimes \mathcal Z_1 \otimes \mathcal Z_j \otimes \mathcal Z_1, & \text{ if $m$ is even},
\end{cases}   
\end{align*}
where $\mathcal Z_j \simeq \br$ is the $j$-th irreducible representation of $S_2$ ($j = 0$ for trivial $S_2$-action and $j=1$ for the transposition action) and the final $\mathcal Z_1$ corresponds to the antipodal action $u \to - u$.
\vs
To find the isotropy subgroup $G_u$ for an element $u \in \mathcal V_{m,j}$, we analyze the condition $g \cdot u = u$ for an element $g = (\gamma, z_{inv}, \sigma, z_{ant}) \in SO(3) \times \bz_2 \times S_2 \times \bz_2$. Clearly, if $g$ fixes $u$, then one has either $u(\gamma x) = u(x)$ or $u(\gamma x) = - u(x)$. Adopting the notation $s : O(3) \to \{\pm 1\}$ for the sign function
\begin{align*}
    s(\gamma) = 
    \begin{cases}
        1 & \text{ if } u(\gamma x) = u(x); \\
        -1 & \text{ if } u(\gamma x) = -u(x),
    \end{cases}
\end{align*}
we can specify the following necessary and sufficient condition for
$(\gamma, \pm I,\sigma,\pm 1) \in G$ with to fix an element $u \in \mathcal V_{m,j}$ with $u(\gamma x) = u(x)$ or $u(\gamma x) = - u(x)$:
\begin{align*}
\begin{cases}
    \sign(\pm I) \sign(\pm 1) = s(\gamma) & \text{ if $m$ is odd and $j = 0$}; \\
    \sign(\pm I) \sign(\sigma) \sign(\pm 1) = s(\gamma) & \text{ if $m$ is odd and $j = 1$}; \\
    \sign(\sigma) \sign(\pm 1) = s(\gamma) & \text{ if $m$ is even and $j = 1$}; \\
    \sign(\pm 1) = s(\gamma) & \text{ if $m$ is even and $j = 0$}. 
\end{cases}
\end{align*}
In particular, if we know $(\gamma,-I) u = u$, for some $u \in \mathcal W_m$, then the following elements of $G$ will fix the same $u \in \mathcal V_{m,j}$:
\begin{table}[htbp] 
\centering 
    \begin{tabular}{|l|l|l|} 
    \hline
    $\mathcal W_{m,j}^-$ & $j=0$ & $j=1$\\
    \hline
    $m$ even & $(\gamma,-I,(12),-1)$ &  $(\gamma,-I,(12),1)$\\
        & $(\gamma,-I,(~~),-1)$ &  $(\gamma,-I,(~~),-1)$\\
    & $(\gamma,I,(12),-1)$ &  $(\gamma,I,(12),1)$\\
    & $(\gamma,I,(~~),-1)$ &  $(\gamma,I,(~~),-1)$\\
    \hline
    $m$ odd & $(\gamma,-I,(12),1)$ &  $(\gamma,-I,(12),-1)$\\
        & $(\gamma,-I,(~~),1)$ &  $(\gamma,-I,(~~),1)$\\
    & $(\gamma,I,(12),-1)$ &  $(\gamma,I,(12),1)$\\
    & $(\gamma,I,(~~),-1)$ &  $(\gamma,I,(~~),-1)$\\
    \hline
    \end{tabular}
    \\[1ex] %
\caption{Elements $(\gamma,\pm I,\sigma ,\pm 1) \in SO(3) \times \mathbb Z_2 \times S_2 \times \mathbb Z_2$ that fix $u \in\mathcal W_{m,j}^-$, given that $(\gamma,-I)\in O(3)$ fixes $u \in \mathcal W_m$. Here, $(~~)\in S_2$ is the neutral element.}
\label{table:generator-correspondence}
\end{table}
Therefore, a natural approach to characterizing the resulting subgroups is to start with the known maximal isotropy subgroups of the $O(3)$-action on $\mathcal W_m$ and determine how they combine with elements from $S_2 \times \bz_2$. Fortunately, the maximal isotropy subgroups of $O(3)$ on spaces of spherical harmonics have already been characterized in \cite{Golubitsky}. Restricting this list to only those which are also maximal on the full lattice of orbit types, we obtain
\begin{align*}
&m =1:& &O(2)^-\\
&m =3:& &O(2)^-, \mathbb O^-\\
&m =5:& &O(2)^-\\
&m = 7,11:& &O(2)^-,\mathbb O^-\\
&m = 9,13,17,19,23,29:& &O(2)^-,\mathbb O^-,\mathbb O\\
&\textit{All other odd $m$}:& &O(2)^-,\mathbb O^-,\mathbb O,\mathbb I\\
&m=2:& &O(2)\\
&m=4,8:& &O(2),\mathbb O\\
&m=14:& &O(2),\mathbb O^-,\mathbb O\\
&\textit{All other even $m$}:& &O(2),\mathbb O^-,\mathbb O,\mathbb I
\end{align*}
Taking each of these orbit types and applying correspondences in Table \ref{table:generator-correspondence}, we obtain the maximal orbit types for $O(3) \times S_2 \times \mathbb Z_2$ with the natural action given in Section \ref{sec:isotypic_decomp}. However, we also know that any orbit type which is contained in $O(3) \times S_2\times \mathbb Z_1$ is considered to be in the radial sublattice and must also be removed. This results in the removal of the equivalent orbit types to $O(2),\mathbb O$, and $\mathbb I$ for even $m$, as these are all contained inside $O(3) \times S_2 \times \mathbb Z_1$. The resulting list of maximal non-radial orbit types is presented, using the abbreviations in Table \ref{table:max-orbit-types-abbrev}, in Table \ref{table:max-orbit-types-sequential}.
\begin{table}[h]
\centering
\begin{tabular}{llllll}
\hline
\text{Abbreviation} & $K_1 \leq O(3)$ & $K_2 \leq S_2 \times \mathbb Z_2$ & $Z_1 \trianglelefteq K_1$ & $Z_2 \trianglelefteq K_2$\\
\hline
$O(2)^-_0$ & $O(2)^p$ & ${S_2}^p$ & $O(2)^-$ & $\langle((12),1)\rangle$\\
$\mathbb O_0$ & $\mathbb O^p$ & ${S_2}^p$ & $\mathbb O$ & $\langle((12),1)\rangle$\\
$\mathbb O^-_0$ & $\mathbb O^p$ & ${S_2}^p$ & $\mathbb O^-$ & $\langle((12),1)\rangle$\\
$\mathbb I_0$ & $\mathbb I^p$ & ${S_2}^p$ & $\mathbb I$ & $\langle((12),1)\rangle$\\
$\mathbb O^+_0$ & $\mathbb O^p$ & ${S_2}^p$ & $\mathbb T^p$ & $\langle((12),1)\rangle$\\
$O(2)^-_1$ & $O(2)^p$ & ${S_2}^p$ & $O(2)^-$ & $\langle((12),-1)\rangle$\\
$\mathbb O_1$ & $\mathbb O^p$ & ${S_2}^p$ & $\mathbb O$ & $\langle((12),-1)\rangle$\\
$\mathbb O^-_1$ & $\mathbb O^p$ & ${S_2}^p$ & $\mathbb O^-$ & $\langle((12),-1)\rangle$\\
$\mathbb I_1$ & $\mathbb I^p$ & ${S_2}^p$ & $\mathbb I$ & $\langle((12),-1)\rangle$\\
$\mathbb O^+_1$ & $\mathbb O^p$ & ${S_2}^p$ & $\mathbb T^p$ & $\langle((12),-1)\rangle$\\
\end{tabular}
\caption{Abbreviations and amalgamated components of maximal non-radial orbit types of $O(3)\times S_2 \times \mathbb Z_2$, written in amalgamated notation as ${K_1}^{Z_1} \times_{\mathbb Z_2} {}^{Z_2}{K_2}$.}
\label{table:max-orbit-types-abbrev}
\end{table}

\begin{table}[h!]
\centering
\begin{tabular}{ll@{\hskip 1.5em}ll@{\hskip 1.5em}ll@{\hskip 1.5em}ll}
\hline
\textbf{$m$} & & \textbf{$m$} & & \textbf{$m$} &  & \textbf{$m$} &  \\
\hline
1  & $O(2)^-_j$                      & 16 & $\mathbb{O}^+_j$ & 31 & $O(2)^-_j,\mathbb O^-_j,\mathbb{I}_j$ & 46 & \\
2  &                                 & 17 & $O(2)^-_j,\mathbb O_j,\mathbb O^-_j$ & 32 &                                 & 47 & $O(2)^-_j,\mathbb O_j,\mathbb{I}_j$ \\
3  & $O(2)^-_j,\mathbb O^-_j$         & 18 &                  & 33 & $O(2)^-_j,\mathbb O_j,\mathbb O^-_j,\mathbb{I}_j$ & 48 & \\
4  &                                 & 19 & $O(2)^-_j,\mathbb O_j$ & 34 & $\mathbb{O}^+_j$ & 49 & $O(2)^-_j,\mathbb{I}_j$ \\
5  & $O(2)^-_j$                      & 20 & $\mathbb{O}^+_j$ & 35 & $O(2)^-_j,\mathbb O^-_j,\mathbb{I}_j$ & 50 & \\
6  & $\mathbb{O}^+_j$                & 21 & $O(2)^-_j,\mathbb{I}_j$ & 36 & $\mathbb{O}^+_j$ & 51 & $O(2)^-_j,\mathbb O^-_j$ \\
7  & $O(2)^-_j,\mathbb O^-_j$         & 22 &                  & 37 & $O(2)^-_j,\mathbb O_j,\mathbb O^-_j,\mathbb{I}_j$ & 52 & \\
8  &                                 & 23 & $O(2)^-_j,\mathbb O_j$ & 38 & $\mathbb{O}^+_j$ & 53 & $O(2)^-_j,\mathbb{I}_j$ \\
9  & $O(2)^-_j,\mathbb O_j,\mathbb O^-_j$ & 24 &                  & 39 & $O(2)^-_j,\mathbb O_j,\mathbb{I}_j$ & 54 & $\mathbb{O}^+_j$ \\
10 & $\mathbb{O}^+_j$                & 25 & $O(2)^-_j,\mathbb{I}_j$ & 40 & $\mathbb{O}^+_j$ & 55 & $O(2)^-_j,\mathbb O^-_j$ \\
11 & $O(2)^-_j,\mathbb O^-_j$         & 26 &                  & 41 & $O(2)^-_j,\mathbb O_j,\mathbb O^-_j,\mathbb{I}_j$ & 56 & \\
12 & $\mathbb{O}^+_j$                & 27 & $O(2)^-_j,\mathbb O^-_j,\mathbb{I}_j$ & 42 &                  & 57 & $O(2)^-_j,\mathbb O_j,\mathbb O^-_j$ \\
13 & $O(2)^-_j,\mathbb O_j,\mathbb O^-_j$ & 28 &                  & 43 & $O(2)^-_j,\mathbb O_j,\mathbb{I}_j$ & 58 & $\mathbb{O}^+_j$ \\
14 & $\mathbb{O}^+_j$                & 29 & $O(2)^-_j$       & 44 & $\mathbb{O}^+_j$ & 59 & $O(2)^-_j,\mathbb O^-_j,\mathbb{I}_j$ \\
15 & $O(2)^-_j,\mathbb O_j,\mathbb{I}_j$   & 30 & $\mathbb{O}^+_j$ & 45 & $O(2)^-_j$       & 60 & $\mathbb{O}^+_j$ \\
\hline
\end{tabular}
\caption{Maximal non-radial orbit types on $\mathcal W_{m,j}$ for $j \in \{0,1\}$. The pattern is periodic in $m$ with period 60. Blank entries indicate no such orbit types appear for that $m$.}
\label{table:max-orbit-types-sequential}
\end{table}
\FloatBarrier
\section{The $G$-Equivariant Degree}\label{sec:appendix}
\noi{\bf  Equivariant notation.}
Let $G$ be a compact Lie group. For any subgroup  $H \leq G$, we denote by $(H)$ its conjugacy class,
by $N(H)$ its normalizer by $W(H):=N(H)/H$ its Weyl group in $G$. The set of all subgroup conjugacy classes in $G$ 
\[
\Phi(G):=\{(H): H\le G\},
\]
and has a natural partial order defined as follows
\[
(H)\leq (K) \iff \exists_{ g\in G}\;\;gHg^{-1}\leq K.
\]
As is possible with any partially ordered set, we extend the natural order over $\Phi(G)$ to a total order, which we indicate by `$\prec$' to differentiate the two relations. Moreover, we put
\[
\Phi_0 (G):= \{ (H) \in \Phi(G) \; : \; \text{$W(H)$  is finite}\},
\]
and, for any $(H),(K) \in \Phi_0(G)$, we denote by $n(H,K)$ the number of subgroups $\tilde K \leq G$ with $\tilde K \in (K)$ and $H \leq \tilde K$. Given a $G$-space $X$ with an element $x \in X$, we denote by
$G_{x} :=\{g\in G:gx=x\}$ the {\bf isotropy group} of $x$
and we call $(G_{x}) \in \Phi(G)$  the {\bf orbit type} of $x \in X$. We also put 
\begin{align*}
    \begin{cases}
       \Phi(G,X) := \{(H) \in \Phi(G)  : 
(H) = (G_x) \; \text{for some $x \in X$}\}, \\
\Phi_0(G,X):= \Phi(G,X) \cap \Phi_0(G).
    \end{cases}
\end{align*}
Given any subgroup $H\leq G$, we call the subspace 
\[
X^{H} :=\{x\in X:G_{x}\geq H\},
\]
the {\bf $H$-fixed-point subspace} of $X$. If $Y$ is another $G$-space, then a continuous map $f : X \to Y$ is said to be {\bf $G$-equivariant} if $f(gx) = gf(x)$ for each $x \in X$ and $g \in G$.
\vs
\noi{\bf The Burnside ring and Axioms of the $G$-Equivariant Brouwer Degree.}
We call the free $\mathbb{Z}$-module $A(G) := \mathbb{Z}[\Phi_0(G)]$ the {\bf Burnside ring} when it is equipped with the multiplicative operation
\begin{align} \label{def:burnside_product}
    (H) \cdot (K) := \sum\limits_{(L) \in \Phi_0(G)} n_L(L), \quad (H),(K) \in \Phi_0(G), 
\end{align}
where the coefficients $n_L \in \mathbb{Z}$ are given by the recurrence formula
\begin{align} \label{def:recurrence_formula_coefficients_burnside_product}
    n_L := \frac{n(L,H) |W(H)| n(L,K) |W(K)| - \sum_{(\tilde L) \succ (L)} n_{\tilde L} n(L,\tilde L) |W(\tilde L)|}{|W(L)|}.
\end{align}
Since every Burnside ring element $a \in A(G)$ can be expressed as a formal sum over some finite number of generator elements  
\[
a = n_{H_1}(H_1) + n_{H_2}(H_2) + \cdots + n_{H_N}(H_N),
\]
we can use the notation
\[
\operatorname{coeff}^H: A(G) \rightarrow \bz, \quad
\operatorname{coeff}^H(a) := n_H,
\]
to specify the integer coefficient standing next to the generator element $(H) \in \Phi_0(G)$.
\vs
Let $V$ be an orthogonal $G$-representation and suppose that $\Om \subset V$ is an open bounded $G$-invariant set. A $G$-equivariant map $f:V \rightarrow V$ is said to be $\Om$-admissible if $f(x) \neq 0$ for all $x \in \partial \Om$, in which case the pair $(f,\Om)$ is called an {\bf admissible $G$-pair} in $V$. We denote by $\mathcal M^{G}(V)$ the set of all admissible $G$-pairs in $V$ and by $\mathcal{M}^{G}$ the set of all admissible $G$-pairs defined by taking a union over all orthogonal $G$-representations, i.e.
\[
\mathcal M^{G} := \bigcup\limits_V \mathcal M^{G}(V).
\]
The $G$-equivariant Brouwer degree provides an algebraic count of solutions, according to their symmetric properties, to equations of the form
\[
f(x) = 0, \; x \in \Omega,
\]
where $(f, \Omega) \in \mathcal M^{G}$. In fact, it is standard (cf. \cite{AED}, \cite{book-new}) to define the {\bf $G$-equivariant Brouwer degree} as the unique map associating to every admissible $G$-pair $(f,\Om)\in \mathcal M^{G}$ an element from the Burnside ring $A(G)$, satisfying the three {\bf degree axioms} of additivity, homotopy and normalization:
\vs
\begin{theorem} \rm
\label{thm:GpropDeg} There exists a unique map $\gdeg:\mathcal{M}
	^{G}\to A(G)$, that assigns to every admissible $G$-pair $(f,\Omega)$ the Burnside ring element
	\begin{equation}
		\label{eq:G-deg0}\gdeg(f,\Omega)=\sum_{(H) \in \Phi_0(G)}%
		{n_{H}(H)},
	\end{equation}
	satisfying the following properties:
\begin{enumerate}[label=($G_\arabic*$)]
\item\label{g1} \textbf{(Additivity)}
For any two  disjoint open $G$-invariant subsets
  $\Omega_{1}$ and $\Omega_{2}$ with
		$f^{-1}(0)\cap\Omega\subset\Omega_{1}\cup\Omega_{2}$, one has
		\begin{align*}
\gdeg(f,\Omega)=\gdeg(f,\Omega_{1})+\gdeg(f,\Omega_{2}).
		\end{align*}
\item\label{g2} \textbf{(Homotopy)}
For any $\Omega$-admissible $G$-homotopy, $h:[0,1]\times V\to V$, one has
\begin{align*}
\gdeg(h_{t},\Omega)=\mathrm{constant}.
\end{align*}
\item\label{g3}\textbf{(Normalization)}
For any open bounded neighborhood of the origin in an orthogonal $G$-representation $V$ with the identity operator $\id:V \rightarrow V$, one has
\begin{align*}
	\gdeg(\id,\Omega)=(G).
\end{align*}
\end{enumerate}
 \vs
The following are additional properties of the map $\gdeg$ which can be derived from the four axiomatic properties defined above (cf. \cite{AED}, \cite{book-new}):		
\begin{enumerate}[label=($G_\arabic*$),start=4]
\item \label{g4} \textbf{(Existence)} If  $n_{H} \neq0$ for some $(H) \in \Phi_0(G)$ in \eqref{eq:G-deg0}, then there
		exists $x\in\Omega$ such that $f(x)=0$ and $(G_{x})\geq(H)$.
		\item \label{g5} {\textbf{(Multiplicativity)}} For any $(f_{1},\Omega
		_{1}),(f_{2},\Omega_{2})\in\mathcal{M} ^{G}$,
		\begin{align*}
			\gdeg(f_{1}\times f_{2},\Omega_{1}\times\Omega_{2})=
		\gdeg(f_{1},\Omega_{1})\cdot \gdeg(f_{2},\Omega_{2}),
		\end{align*}
		where the multiplication `$\cdot$' is taken in the Burnside ring $A(G )$.
		\item \label{g6} \textbf{(Recurrence Formula)} For an admissible $G$-pair
		$(f,\Omega)$, the $G$-degree \eqref{eq:G-deg0} can be computed using the
		following Recurrence Formula:
		\begin{equation}
			\label{eq:RF-0}n_{H}=\frac{\deg(f^{H},\Omega^{H})- \sum_{(K)\succ(H)}{n_{K}\,
					n(H,K)\, \left|  W(K)\right|  }}{\left|  W(H)\right|  },
		\end{equation}
		where $\left|  X\right|  $ stands for the number of elements in the set $X$
		and $\deg(f^{H},\Omega^{H})$ is the Brouwer degree of the map $f^{H}%
		:=f|_{V^{H}}$ on the set $\Omega^{H}\subset V^{H}$.
	\end{enumerate}
\end{theorem}
The natural generalization of the $G$-equivariant Brouwer degree to its infinite dimensional counterpart the $G$-equivariant Leray-Schauder degree is described in detail elsewhere (see for example \cite{survey, book-new, AED}).
\vs
\noi{\bf Computational Formulae for the $G$-Equivariant Brouwer Degree.} 
We denote by $\{ \mathcal V_i \}_{i \in \mathbb{N}}$ the set of all irreducible $G$-representations and define the $i$-th basic degree as follows
\begin{align*}
\deg_{\mathcal{V}_{i}}:=\gdeg(-\id,B(\mathcal{V} _{i})).
\end{align*} 
Given any orthogonal $G$-representation with a $G$-isotypic decomposition
\[
V = \bigoplus_{i \in \mathbb{N}} V_i,
\]
and any $G$-equivariant  linear isomorphism $T:V\to V$, the Multiplicativity and Homotopy properties of the $G$-equivariant Brouwer degree, together with Schur's Lemma implies
\begin{align*}
  \gdeg(T,B(V))=\prod_{i \in \mathbb{N}} \gdeg
	(T_{i},B(V_{i}))= \prod_{i \in \mathbb{N}}\prod_{\mu\in\sigma_{-}(T)} \left(
	\deg_{\mathcal{V} _{i}}\right)  ^{m_{i}(\mu)}%
\end{align*}
where $T_{i}=T|_{V_{i}}$, $\sigma_{-}(T)$ denotes the real negative
spectrum of $T$ and $m_i(\mu) := \dim E_i(\mu)/ \dim \mathcal W_i$ (here, we indicate by
$E(\mu)$ the generalized eigenspace associated with any $\mu \in \sigma(T)$ and $E_i(\mu) := E(\mu) \cap V_i$).
\vskip.3cm
Notice that each of the basic degrees: 
\begin{align*}
	\deg_{\mathcal{V} _{i}}=\sum_{(H)}n_{H}(H),
\end{align*}
can be practically computed, using the recurrence formula  \eqref{eq:RF-0}, as follows
\begin{align}\label{eq:RF_bd}
n_{H}=\frac{(-1)^{\dim\mathcal{V} _{i}^{H}}- \sum_{(H)\prec(K)}{n_{K}\, n(H,K)\, \left|  W(K)\right|  }}{\left|  W(H)\right|  }.
\end{align}


\begin{thebibliography}{99}
\normalsize
\baselineskip=17pt


\bibitem{BalHooton}  Z. Balanov, E. Hooton, W. Krawcewicz, \& D. Rachinskii, {\it Patterns of non-radial solutions to coupled semilinear elliptic systems on a disc}, Nonlinear Analysis, Vol. 202 (2021) No. 112094, https://doi.org/10.1016/j.na.2020.112094 

\bibitem{BalBurnett} Z. Balanov,  J.V. Burnett, W. Krawcewicz, H. Xiao,  {\it Global Bifurcation of Periodic Solutions in Symmetric Reversible Second Order Systems with Delays}, International Journal of Bifurcation and Chaos, Vol. 31 (2021) No. 12.

\bibitem{BalChen} Z. Balanov, F. Chen, J. Guo,  W. Krawcewicz,  {\it Periodic solutions to reversible second order autonomous systems with commensurate delays}, Topological Methods in Nonlinear Analysis, Vol. 59 (2021) No. 2A, pp. 475–498.

\bibitem{survey}  Z. Balanov, W. Krawcewicz,   S. Rybicki, H. Steinlein,  {\it A short treatise on the equivariant degree theory and its applications}, Journal of Fixed Point Theory and Applications, Vol. 8 (2010), pp. 1-74, https://doi.org/10.1007/s11784-010-0033-9.

\bibitem{book-new} Z. Balanov, W. Krawcewicz,  D. Rachinskii, J. Yu, H-P. Wu, {\it Degree Theory and Symmetric Equations Assisted by GAP System: With a Special Focus on Systems with Hysteresis}, AMS Mathematical Surveys and Monographs, Vol. 286 (2025).

\bibitem{Balkraw} Z. Balanov, W. Krawcewicz, A. Murza, {\it Periodic solutions to reversible second order autonomous DDEs in prescribed symmetric nonconvex domains}, Nonlinear Differ. Equ. Appl. 28, 40 (2021), https://doi.org/10.1007/s00030-021-00695-7.


\bibitem{AED} Z. Balanov, W. Krawcewicz, H. Steinlein, {\it Applied Equivariant Degree}, AIMS Series on Differential Equations \& Dynamical Systems, Vol. 1 (2006).

\bibitem{Bruckler} F.M. Brückler, V. Stilinović, \emph{From Friezes to Quasicrystals: A History of Symmetry Groups} In: Sriraman, B. (eds) Handbook of the History and Philosophy of Mathematical Practice. Springer, Cham (2023). https://doi.org/10.1007/978-3-030-19071-2


\bibitem{Cerami} G. Cerami, \emph{Symmetry breaking for a class of semi-linear elliptic problems}, (MRC Technical Summary Report \#2759), University of Wisconsin-Madison, Mathematics Research Center (1984).


\bibitem{Dancer} E. N. Dancer, \emph{Some notes on the method of moving planes}, Bulletin of the Australian Mathematical Society, 46(3), 425–434 (1992). https://doi.org/10.1017/s0004972700012089

\bibitem{DuanCrane} Y. Duan, C. Crane, W. Krawcewicz, H. Xiao, \emph{Periodic solutions in reversible symmetric second order systems with multiple distributed delays}, Journal of Differential Equations, 401 (2024), 282–307. https://doi.org/10.1016/j.jde.2024.04.030 


\bibitem{Duan} Y. Duan,  W. Krawcewicz, H. Xiao, {\it Periodic solutions in reversible systems in second order systems with distributed delays}, AIMS Mathematics, Vol. 9 (2024) Issue 4, pp. 8461–8475, https://doi.org/10.3934/math.2024411.

\bibitem{Eze} I. Eze, C. García-Azpeitia,  W. Krawcewicz,  Y. Lv,  {\it Subharmonic solutions in reversible non-autonomous differential equations}, Nonlinear Analysis, Vol. 216 (2022), pp. 112675–112675, https://doi.org/10.1016/j.na.2021.112675.

\bibitem{Carlos} C. García-Azpeitia,  W. Krawcewicz, Y. Lv,  {\it Solutions of fixed period in the nonlinear wave equation on networks}, Nonlinear Differential Equations and Applications NoDEA, Vol. 26 (2019) No. 4, https://doi.org/10.1007/s00030-019-0568-4.

\bibitem{Carloskraw} C. García-Azpeitia,  W. Krawcewicz, M. Tejada-Wriedt, H-P. Wu  {\it Global nonlinear normal modes in the fullerene molecule $C_{60}$}, SIAM Journal on Applied Dynamical Systems, Vol. 20 Iss. 1 (2021), https://doi.org/10.1137/19M1269865.


\bibitem{Ghanem1} Z. Ghanem,  {\it Nonstationary solutions to symmetric systems of second-order differential equations}, Annales Polonici Mathematici, (2025), https://doi.org/10.4064/ap240128-3-3 

\bibitem{Ghanem} Z. Ghanem,  C. Crane,  J. Liu,  {\it Global bifurcation of non-radial solutions for symmetric sub-linear elliptic systems on the planar unit disc}, Journal of Fixed Point Theory and Applications, Vol. 26 (2024) No. 4, https://doi.org/10.1007/s11784-024-01133-8.

\bibitem{Gidas} B. Gidas,  W. Ni \& L. Nirenberg, \emph{Symmetry and related properties via the maximum principle}, Communications in Mathematical Physics, 68(3), 209–243 (1979). https://doi.org/10.1007/bf01221125

\bibitem{Golubitsky} Golubitsky M, Stewart I, Schaeffer D G. {\sf Singularities and Groups in Bifurcation Theory}, Vol.2, Springer-Verlag, New York, (1998).


\bibitem{Krawcewicz} W. Krawcewicz, H. Wu, S. Yu, {\it Periodic solutions in reversible second order autonomous systems with symmetries}, Journal of Nonlinear and Convex Analysis, Vol. 18 (2017) No. 8, pp. 1393-1419.


\bibitem{Jingzhou} J. Liu, C. García-Azpeitia, \& W. Krawcewicz, \emph{Existence of non-radial solutions to semilinear elliptic systems on a unit ball in $\mathbb R^3$}, J. Fixed Point Theory Appl. 25, 86 (2023). https://doi.org/10.1007/s11784-023-01086-4


\bibitem{Serrin} J. Serrin, \emph{A symmetry problem in potential theory}, Archive for Rational Mechanics and Analysis, 43(4), 304–318 (1971). https://doi.org/10.1007/bf00250468

\bibitem{Smoller} J. Smoller \& A. Wasserman,   \emph{Symmetry-breaking for positive solutions of semilinear elliptic equations}, Archive for Rational Mechanics and Analysis, 95(3), 217–225 (1986). https://doi.org/10.1007/bf00251359


\bibitem{Vanderbauwhede} A. Vanderbauwhede, \emph{Local bifurcation and symmetry}, Pitman Advanced Publishing Program (1982).

\bibitem{Golubitsky-book} M. Golubitsky, I. Stewart, D. G. Schaeffer, \emph{Singularities and Groups in Bifurcation Theory: Volume II}, Springer New York, (1988)
‌
\end{thebibliography}
\end{document}